\documentclass[reqno, 12pt]{amsart}
\usepackage{amsmath,amsthm,amsfonts,amssymb}

\date{}
\newtheorem{theorem}{Theorem}[section]
\newtheorem{lemma}[theorem]{Lemma}
\newtheorem{corollary}[theorem]{Corollary}
\newtheorem{remark}[theorem]{Remark}
\newtheorem{proposition}[theorem]{Proposition}

\numberwithin{equation}{section}

\begin{document}

\centerline{{\bf  Kantorovich problems and conditional measures depending on a parameter}}

\vskip .2in

\centerline{{\bf Vladimir I. Bogachev$^{a,b,}$\footnote{Corresponding author, vibogach@mail.ru.},
Ilya I. Malofeev$^{a}$}}

\vskip .2in

$^{a}$ Department of Mechanics and Mathematics,
Moscow State University,
119991 Moscow, Russia

$^{b}$  National Research University Higher School of Economics, Moscow, Russia

\vskip .2in

{\bf Abstract.}
We study measurable dependence of measures on a parameter in the following two
classical problems:   constructing conditional measures and the Kantorovich
optimal transportation. For parametric families of measures and mappings
we prove the existence of conditional measures measurably depending on the parameter.
A~particular emphasis is made
on the Borel measurability (which cannot be always achieved).
Our second main result  gives sufficient conditions for the Borel measurability
of optimal transports and transportation costs with respect to a parameter in the case where
marginal measures and cost functions depend on a parameter.
As a corollary we obtain the Borel measurability with respect to the
parameter for disintegrations of optimal plans.
Finally, we show that the Skorohod parametrization of measures by mappings can be also
made measurable with respect to a parameter.

\vskip .1in


\vskip .1in


\section{Introduction}

We recall that,  given two probability spaces
$(X,\mathcal{B}_X,\mu)$ and $(Y,\mathcal{B}_Y, \nu)$ and a nonnegative
$\mathcal{B}_X\otimes \mathcal{B}_Y$-measurable function $h$
on $X\times Y$ (called a cost function), the  associated Kantorovich problem is to find
the infimum of the integral
$$
I_h(\sigma):=\int h\, d\sigma
$$
over all probability measures  $\sigma$ on $\mathcal{B}_X\otimes \mathcal{B}_Y$ with projections
$\mu$ and $\nu$ on the factors.
This infimum is denoted by
$$
K_h(\mu,\nu)
$$
and called the transportation cost for $h,\mu, \nu$. If this infimum
is attained (is a minimum, which happens under broad assumptions),
then the minimizing measures are called optimal measures (and
also optimal plans or optimal transports). The measures $\mu$ and $\nu$ are
called marginal distributions. There is an extensive literature on this
subject, see, e.g.,
\cite{AG}, \cite{BK}, \cite{GM}, \cite{RR}, \cite{V03}, and~\cite{V}.
This paper was motivated by several questions posed by Sergey Kuksin about measurable
dependence of Kantorovich optimal transportation plans on a parameter in optimal transportation
problems depending on a parameter.

Suppose now that $(T,\mathcal{T})$ is a measurable space and for each $t$
we have marginal probability measures $\mu_t$ and $\nu_t$ (which depend on $t$
measurably in the sense that the functions $t\mapsto \mu_t(A)$
 are $\mathcal{T}$-measurable for all $A\in \mathcal{B}_X$ and similarly for~$\nu_t$)
 and that also the cost function depends on the parameter~$t$, i.e.,
 $$
 h\colon T\times X\times Y\to [0,+\infty)
 $$
is a $\mathcal{T}\otimes\mathcal{B}_X\otimes\mathcal{B}_Y$-measurable
function.
We set
$$
h_t(x,y):=h(t,x,y).
$$
Thus, we obtain a Kantorovich problem with a parameter. Dependence on a parameter
appears even for a single cost function if only marginal distributions  depend on~$t$.
The question is whether the infimum depends measurably on $t$ and there are optimal
plans $\sigma_t$ measurably depending on $t$.

Several results have already been obtained in this situation.
Villani \cite{V} considered the situation where only the marginal
distributions depend on a parameter (and are Borel measures on Polish spaces),
but the cost function does not.
  Dedecker, Prieur and Raynaud De Fitte  \cite{DePR} studied the case
  of metric-type cost functions (such that $h(x,y)=\sup |u(x)-u(y)|$,
  where
  $\sup$ is taken over bounded continuous functions $u$ with $|u(x)-u(y)|\le h(x,y)$)
  on rather general spaces (including completely regular
  Souslin spaces) and established
  the existence of a measurable selection of an optimal measure
   and the measurability of the Kantorovich
  minimum, however, this measurability is with respect
  to the $\sigma$-algebra of universally measurable sets, not with respect to the Borel
  $\sigma$-algebra. Similar results are also contained
  in  \cite[Sections~3.4 and~7.1]{CRFV}.
  Zhang~\cite{Zhang} gave a result for continuous
  cost functions on Polish spaces $X$ and $Y$ and an arbitrary measurable space~$T$,
  but the justification contains a gap and the really proved fact is this: if we consider
  the space $M=C(X\times Y)$ with the Borel $\sigma$-algebra corresponding to the metric
  $$
  d_M(f,g)=\sum_{n=1}^\infty 2^{-n} \min(1, \sup_{z\in B_n} |f(z)-g(z)|),
  $$
  where $\{B_n\}$ is a fixed sequence of increasing balls with the union $X\times Y$,
  and to every triple $(h,\mu,\nu)$ with a nonnegative function $h\in M$ we associate
  the set ${\rm Opt}(h,\mu,\nu)$ of all optimal measures, then
  there is a selection of an optimal measure measurable with respect to the $\sigma$-algebras
  $\mathcal{B}(M)\otimes \mathcal{B}(\mathcal{P}(X))\otimes \mathcal{B}(\mathcal{P}(Y))$ and
  $\mathcal{B}(\mathcal{P}(X\times Y))$. However, this does not imply the measurability
  claimed in~\cite{Zhang} (the measurability with respect to~$\mathcal{T}$
  for a general $\sigma$-algebra~$\mathcal{T}$), because the mapping $t\mapsto h(t,\cdot, \cdot)$ can fail
  to be measurable with respect to $\mathcal{T}$ and $\mathcal{B}(M)$ under the only assumption
  that $h$ is measurable on $T\times X\times Y$. The point is that for a noncompact space $Z$
  the Borel $\sigma$-algebra of the space $C_b(Z)$ with its sup-norm is not generated
  by evaluation functionals $z\mapsto f(z)$ (see Remark~\ref{rem-Zh} below). A~consequence of this in the situation
  of \cite{Zhang} is that the assumed measurability of the cost function is not sufficient
  for the applicability of the established selection result. However, it will be shown below
  in Theorem~\ref{tmain2}
  that the main result of \cite{Zhang} is valid. Moreover, we show that optimal transports
  can be made Borel measurable with respect to the parameter for lower semicontinuous cost
  functions in place of continuous ones, provided that $T$ is a Souslin space with its
  Borel $\sigma$-algebra.

  In the study of optimal plans one often deals with conditional measures. It is, of course, a
  question of independent interest to study conditional measures depending on a parameter
  (and this question was also suggested to us by Sergey Kuksin).
  The general framework for conditional measures is this: given a measure $\mu$
  on a space $X$ and a measurable mapping $f$ of $X$ onto another measurable space~$Y$, we are looking for
  measures $\mu^y$ concentrated on the level sets $f^{-1}(y)$ for $y\in Y$ such that $\mu$ has the form
  $$
\mu=  \int_Y \mu^y\, \nu(dy) ,
  $$
  where $\nu=\mu\circ f^{-1}(dy)$ is the image of $\mu$ under $f$ (or some other natural measure on~$Y$).
  Below we recall a precise definition.
  Again, once $\mu$ and $f$ depend on a parameter $t$, the question is whether one can pick
  conditional measures $\mu^y_t$ measurably depending on~$t$. A positive result was obtained
  in \cite{Malofeev} (where a sketch of the proof was given),
  but, as above, this result is in terms of measurability with respect to the extensions
  of Borel $\sigma$-algebras generated by Souslin sets. Below (see Theorem~\ref{t1})
   we provide all technical details for a more general result and complement this
  result by sufficient conditions for the Borel measurability (Theorem~\ref{t2}).
    Moreover, the existence of
  jointly (i.e., in both variables)
  Borel measurable conditional measures depending on a parameter is shown (see Proposition~\ref{p1})
  to be
  equivalent to the existence of jointly Borel measurable right inverse mappings, similarly
  to the result of Blackwell and Ryll-Nardzewski \cite{BR} in the
   case of measures and mappings without parameters.
   It is worth noting that although sets from
   the $\sigma$-algebra generated by Souslin sets remain measurable with respect to all Borel measures,
   their weak point   is that continuous images (say, projections)
    of such sets can fail to be measurable. This is one of motivations to desire
        the Borel measurability.

   Both problems (dependence on a parameter for optimal plans and conditional  measures)
   have some common features and are strongly connected with measurable choice theorems.
   It will be more convenient to start with conditional measures, which is done in Section~3.
In    Section~4 we discuss optimal plans and formulate our main results, which are proved
in Section~5 along with a number of auxiliary results.

Finally, in Section~6 we consider along the same lines the classical result going back
to Skorohod and giving a parametrization of Borel probability measures $\mu$ on a Polish
space $X$ by Borel mappings $\xi_\mu\colon [0,1]\to X$ such that $\mu$ is the image
of Lebesgue measure $\lambda$ under~$\xi_\mu$ and measures $\mu_n$ converge weakly
to $\mu$ if and only if the mappings $\xi_{\mu_n}$ converge to $\xi_\mu$ almost surely.
We show that there is a version of $\xi_\mu$ such that the function $(\mu,t)\mapsto \xi_\mu(t)$
is Borel measurable on $[0,1]\times \mathcal{P}(X)$. It follows that for any family
of measures $\mu_\omega$ measurably depending on a parameter~$\omega$, the function
$(\omega,t)=\xi_{\mu_\omega}(t)$ is jointly Borel measurable.

\section{Notation and terminology}

We shall consider Borel measures on complete separable metric spaces and in some
results on Souslin spaces. So we briefly recall these concepts and some related objects.

Let $X$ be a topological space. Its Borel $\sigma$-algebra, denoted by~$\mathcal{B}(X)$, is the smallest $\sigma$-algebra
containing all open sets. A~real function $f$ on $X$ is called
Borel measurable if the sets $\{x\colon f(x)<c\}$ are Borel for all~$c$.
A~mapping $f$ from $X$ to a topological space $Y$ is called Borel measurable if $f^{-1}(B)$ is a Borel set
for every Borel set $B\subset Y$. For $Y=\mathbb{R}$ this is equivalent to the aforementioned definition.

The space of bounded continuous functions on~$X$ with its sup-norm
is  denoted by $C_b(X)$.
The space of bounded Borel measurable functions with the same norm is denoted by~$B_b(X)$.

If $(T,\mathcal{T})$ is a measurable space (i.e., $\mathcal{T}$ is a $\sigma$-algebra),
then a mapping $f\colon T\to X$ is called $\mathcal{T}$-measurable (or
$(\mathcal{T},\mathcal{B}(X))$-measurable) if $f^{-1}(B)\in \mathcal{T}$
for all $B\in\mathcal{B}(X)$. The Borel measurability is a particular case of this definition.

A space homeomorphic to a complete separable
metric space is called Polish.
A~Hausdorff space that is the  image of a complete separable
metric space under a continuous mapping is called  Souslin or analytic
(see, e.g., \cite{B07},~\cite{Kech}). If such a mapping can be found one-to-one, then $X$ is called
a Luzin space.
A~Hausdorff space $X$
is completely regular if for every point $x\in X$ and every open set $U$ containing~$x$ there
is a continuous function $f\colon X\to [0,1]$ such that $f(x)=1$ and $f=0$ outside~$U$.

Borel sets in Polish spaces
are Souslin (and even Luzin)
spaces; Borel sets in  Souslin spaces are also Souslin. However, unlike the case of
Borel sets, the complement of a Souslin set $A$ in a Polish space is not always Borel, moreover,
it can be Borel only if $A$ itself is Borel. For this reason, the $\sigma$-algebra $\sigma(\mathcal{S}(X))$
generated by the class $\mathcal{S}(X)$ of all Souslin  sets in $X$ is much larger
than the Borel $\sigma$-algebra
(although its cardinality is the continuum for infinite spaces);
for example, in typical cases it is not countably generated (see \cite[Example~6.5.9]{B07}).

Souslin sets belong to the Lebesgue completion of the Borel $\sigma$-algebra
for every Borel measure on a Souslin space (i.e., they are universally measurable), hence
the same is true for the generated $\sigma$-algebra $\sigma(\mathcal{S}(X))$.
However, this $\sigma$-algebra is not stable under the Souslin operation, unlike
the completion of the Borel $\sigma$-algebra (see \cite[p.~66]{B07}) and unlike
the $\sigma$-algebra
 of universally measurable sets. The images and preimages
of Souslin sets under Borel mappings are Souslin. For Borel sets, only preimages are Borel:
it was shown by Souslin that
the projection of a Borel set in $\mathbb{R}^2$ can fail to be Borel.
Next, the preimages of sets in $\sigma(\mathcal{S}(X))$ under Borel mappings
 are also in~$\sigma(\mathcal{S}(X))$.
This is not true for their images even under continuous mappings: the projection of
the complement of a Souslin set need not belong to   $\sigma(\mathcal{S}(X))$
(for example, the projection of the complement of a Souslin set need not be Lebesgue
measurable).

Borel measures are finite (possibly, signed)
measures on $\mathcal{B}(X)$. A signed Borel measure $\mu$
can be written as $\mu=\mu^{+}-\mu^{-}$, where $\mu^{+}$ and $\mu^{-}$ are mutually
singular nonnegative Borel measures. The measure $|\mu|=\mu^{+}+\mu^{-}$ is called
the total variation of~$\mu$ and the number $\|\mu\|=|\mu|(X)$ is called the total
variation norm or the variation norm. We mostly deal with probability measures.

A Borel measure $\mu$ is called Radon if for every Borel set $B$ and every $\varepsilon>0$
there is a compact set $K_\varepsilon\subset B$ such that $|\mu|(B\backslash K_\varepsilon)<\varepsilon$.
On a Souslin space all Borel measures are Radon.

The image of a measure $\mu$ on $X$
under a measurable mapping $f\colon X\to Y$  is denoted by
$\mu\circ f^{-1}$ and defined by  the equality
$$
(\mu\circ f^{-1})(E)=\mu(f^{-1}(E)), \ E\in \mathcal{B}(Y).
$$

Let $\mathcal{P}(X)$ be the space of all Borel probability measures
on a completely regular space~$X$.
Recall that the weak topology on the whole space $\mathcal{M}(X)$
 of signed Borel measures is generated by duality with
$C_b(X)$, i.e.,  is defined by means of seminorms
$$
\mu\mapsto \biggl|\int_X f\, d\mu\biggr|,
$$
where $f\in C_b(X)$. Throughout the spaces of measures will be
considered with the weak topology and the corresponding Borel structure.

If $X$ is a completely regular Souslin space,
then $\mathcal{M}(X)$ and $\mathcal{P}(X)$ are also completely
regular Souslin spaces; if $X$ is a Polish space, then
$\mathcal{P}(X)$ is also Polish (but $\mathcal{M}(X)$ is not in nontrivial cases)
 and if $X$ is a Luzin space, then so is $\mathcal{P}(X)$. These facts can be found
in  \cite[Chapter~8]{B07} or in \cite[Chapter~5]{B18}.

We shall employ Prohorov's condition for compactness in~$\mathcal{M}(X)$: a~set
$M$ has compact closure in $\mathcal{M}(X)$ if it is bounded in variation
and uniformly tight, which means
that for every $\varepsilon>0$ there is a compact set $K\subset X$ such that
$|\mu|(X\backslash K)\le\varepsilon$ for all~$\mu\in M$. If $X$ is a Polish space, then
this condition is also necessary.

For a completely regular
Souslin space $X$,
a mapping $m\colon\, (\Omega,\mathcal{E}) \to \mathcal{P}(X)$
from a measurable space $(\Omega,\mathcal{E})$ is measurable if and only if
 all functions
$$
\omega\mapsto \int_X \varphi(x)\, m(\omega)(dx), \quad \varphi\in C_b(X)
$$
are  $\mathcal{E}$-measurable.
This is also equivalent to the $\mathcal{E}$-measurability of
all functions
$$
\omega\mapsto \int_X \varphi_n(x)\, m(\omega)(dx)
$$
for any countable family  of functions $\varphi_n\in C_b(X)$ of the form
$\varphi_n=p(f_1,\ldots, f_k)$, where $p$ is a
polynomial on $\mathbb{R}^k$ with rational coefficients and $\{f_j\}\subset C_b(X)$ is
a sequence separating
the points in~$X$ (such sequences exist for all completely regular Souslin spaces,
see \cite[Theorem~6.7.7]{B07}). Recall that any sequence of Borel functions separating
points of a Souslin space generates the Borel $\sigma$-algebra of this space
(see~\cite[Theorem~6.8.9]{B07}).
It is readily verified by the monotone class theorem (see \cite[Theorem~6.7.7]{B07} and \cite{Malofeev})
that this measurability is  equivalent to the
$\mathcal{E}$-measurability of all functions
$$
\omega\mapsto m(\omega)(B), \ B\in \mathcal{B}(X).
$$

Recall that a mapping $\Psi$ from a measurable space
$(T,\mathcal{T})$ to the set of nonempty subsets of a topogical space $X$ is
called measurable if for every open set $U\subset X$ the set
$\{t\colon \Psi(t)\cap U\not=\emptyset\}$ belongs to~$\mathcal{T}$.

The space $\mathcal{K}(X)$ of nonempty compact subsets
of a complete metric space $X$ is equipped with the Hausdorff distance
$$
d_H(A,B)=\inf\{r>0\colon {\rm dist}(a,B)<r, \
{\rm dist}(b,A)<r \ \forall\, a\in A, b\in B\}.
$$
It is known that this space is complete and separable (and is compact if $X$ is compact),
see~\cite{CV}.

\section{Conditional measures depending on a parameter}

We first address the problem of conditional measures. A general discussion can be found
in \cite[Chapter~10]{B07}; see also \cite{AY},
\cite{B10}, \cite{HT}, \cite{H}, \cite{H94}, \cite{Rao}, and~\cite{Tjur}.
Connections between conditional measures and surface measures are considered in~\cite{BM} and~\cite{B17}.

It is known that,  whenever $\mu$ is a  Borel probability measure  on
a Souslin space~$X$
and  $f$ is a Borel mapping from $X$ to a Souslin space $Y$,
the level sets  $f^{-1}(y)$ can be equipped with Borel
 probability measures $\mu^y$,
called conditional measures generated by~$f$, possessing
the following three properties:

1) the measure $\mu^y$ is concentrated on the set $f^{-1}(y)$ for each $y\in f(X)$, i.e.,
$$
\mu^y(f^{-1}(y))=1, \ y\in f(X),
$$

2) the functions
$$
y\mapsto \mu^y(B), \quad B\in\mathcal{B}(X),
$$
are measurable with respect to the $\sigma$-algebra $\sigma(\mathcal{S}(X))$
generated by the class of Souslin sets in~$X$,

3) if  $B\subset X$ and $E\subset Y$ are  Borel sets, then
$$
\mu(B\cap f^{-1}(E))=\int_E \mu^y(B)\, \mu\circ f^{-1}(dy).
$$

Conditional measures with properties 1)--3) are called regular proper conditional
probabilities, the term ``proper'' refers to condition~1).

It should be noted that due to condition 1)
the last equality for all $B$ is equivalent to its special
case with $E=Y$:
$$
\mu(B)=\int_Y \mu^y(B)\, \mu\circ f^{-1}(dy).
$$
Indeed, replacing $B$ by $B\cap f^{-1}(E)$ we have
$\mu^y(B\cap f^{-1}(E))=\mu^y(B)$ if $y\in E$, because $\mu^y$ is concentrated on~$f^{-1}(y)$.
If $y\not\in E$,  then for the same reason $\mu^y(B\cap f^{-1}(E))=0$, since
$f^{-1}(E)\cap f^{-1}(y)=\emptyset$. However, the equivalent formulation with $E$ is sometimes useful.

The equality in condition 3) is equivalent to the following:
for every bounded Borel function $\varphi$
on $X$ and every Borel set $E\subset Y$ we have
$$
\int_{f^{-1}(E)} \varphi\, d\mu=\int_E \int_X \varphi(x)\, \mu^y(dx)\, \mu\circ f^{-1}(dy)
=\int_{f^{-1}(E)} \int_X \varphi(x)\, \mu^{f(u)}(dx)\, \mu(du).
$$
As above, it suffices to have this identity for $E=Y$.

It is known (see \cite{B07}) that  conditional measures are unique in the following sense:
two such collections coincide for all points  $y$ outside a set of measure zero with respect to
the induced measure $\mu\circ f^{-1}$.

If $\mathcal{B}^f=\{f^{-1}(A)\colon A\in \mathcal{B}(Y)\}$ is the $\sigma$-algebra generated by~$f$, then
the function
$$
E(\varphi|\mathcal{B}^f)(u)=\int_X \varphi(x)\, \mu^{f(u)}(dx)
$$
serves as the conditional expectation of $\varphi$ with respect to $\mathcal{B}^f$.

It is possible to modify conditions 1) and 2) as follows:  the Borel measurability of all functions
in~2) can be achieved at the expense of weakening condition~1) by replacing it by the condition that
$\mu^y(f^{-1}(y))=1$ for $\mu\circ f^{-1}$-almost all~$y$. However, in the general case
it is impossible to guarantee  the Borel measurability of all functions $y\mapsto \mu^y(B)$
if the equality $\mu^y(f^{-1}(y))=1$ must hold for each~$y$.
There are counter-examples even in the case where $X$ is a Borel set in $[0,1]$ and $f$ is a smooth
function (see \cite{BR} or \cite[V.~2, p.~430]{B07}). According to~\cite{BR},
if $X$ is a Polish space, the existence of conditional measures $\mu^y$ such that 1) holds and
all functions $y\mapsto \mu^y(B)$ are Borel implies that $f(X)$ is a Borel set.
A~necessary and sufficient condition for the existence of such conditional measures
is this: there exists a mapping $F\colon\, X\to X$ measurable with respect to
$\mathcal{B}^f$ and $\mathcal{B}(X)$ such that $f(F(x))=f(x)$.
If $f$ is surjective, this is equivalent to the existence of a Borel mapping
$g\colon Y\to X$ that is right inverse to~$f$: $f(g(y))=y$. Indeed, since $F$ is
$\mathcal{B}^f$-measurable, it must be of the form $F(x)=g(f(x))$ for some Borel mapping
$g\colon Y\to X$, hence $f(g(f(x)))=f(x)$, whence $f(g(y))=y$ for all $y\in Y$. Conversely,
if such $g$ exists, we can take $F(x)=g(f(x))$.
Some measurability problems connected with conditional measures are discussed in~\cite{Ram}.

Suppose  now that $\mu$ and $f$   depend
measurably on a parameter $z$ belonging to some Souslin space~$Z$. Is it possible to pick
conditional measures $\mu_z^y$ depending measurably on~$(y,z)$?
This question arises naturally in applications,
in particular,  in optimal transportation and parametric statistics
(see, e.g., \cite{Pf}, \cite{BK}, \cite{Ev}, and~\cite{V}).
Some positive results have been recently given in \cite{Malofeev}. Here we reinforce these results
(and also give all details of proofs omitted in~\cite{Malofeev}.

Throughout that  $X, Y, Z$ are assumed to be completely regular Souslin spaces
(in some results certain stronger assumptions are used).

\begin{lemma}\label{lem2.1}
Suppose that $\psi\colon\, X\times Z\to \mathbb{R}$ is a bounded function
measurable with respect to the  $\sigma$-algebra
$\sigma(\mathcal{S}(X))\otimes \sigma(\mathcal{S}(Z))$.
Let $z\mapsto \mu_z$, $X\to  \mathcal{P}(X)$
be Borel measurable or, more generally,
$(\sigma(\mathcal{S}(Z)),\mathcal{B}(\mathcal{P}(X)))$-measurable.
Then the function
$$
h(z)=\int_X \psi(x,z)\, \mu_z(dx)
$$
is $\sigma(\mathcal{S}(Z))$-measurable  on $Z$.
If $\psi$ and $z\mapsto \mu_z$ are Borel measurable, then the function $h$ is also Borel measurable.
\end{lemma}
\begin{proof}
In the case of the $(\sigma(\mathcal{S}(Z)),\mathcal{B}(\mathcal{P}(X)))$-measurability
the class $\mathcal{H}$ of all boun\-ded $\sigma(\mathcal{S}(X))\otimes \sigma(\mathcal{S}(Z))$-measurable functions $\psi$
for which $h$ is $\sigma(\mathcal{S}(Z))$-measurable is closed with respect to uniform limits and
limits of increasing uniformly bounded sequences.
Moreover, it contains all functions of the form
$$
\varphi_1(x)\psi_1(z)+\cdots +\varphi_n(x)\psi_n(z),
$$
where   $\varphi_i$ and $\psi_i$ are bounded functions
on $X$ and $Z$ measurable with respect to $\sigma(\mathcal{S}(X))$ and
$\sigma(\mathcal{S}(Z))$, respectively. Applying the monotone class theorem, we conclude that $\mathcal{H}$
is the space of all bounded $\sigma(\mathcal{S}(X))\otimes \sigma(\mathcal{S}(Z))$-measurable functions
(see \cite[Theorem~2.12.9]{B07}). In the  case of Borel measurability, the same reasoning applies if we take for $\mathcal{H}$
 the class  of all bounded Borel measurable functions for which the corresponding function $h$ is Borel measurable.
\end{proof}

\begin{remark}\label{rem3.2}
\rm
It follows from the lemma that if we have a family of Borel sets $B_z$ such that
the function $I_{B_z}(x)$ is $\sigma(\mathcal{S}(X))\otimes \sigma(\mathcal{S}(Z))$-measurable,
then the function
$z\mapsto \mu_z(B_z\cap B)$ is $\sigma(\mathcal{S}(Z))$-measurable
for all $B\in\sigma(\mathcal{S}(X))$ and similarly in the Borel case.
\end{remark}

\begin{lemma}\label{lem2.2}
Suppose that we  have a mapping $(y,z)\mapsto \nu_z^y$ from $Y\times Z$
to $\mathcal{P}(X)$ that is measurable with respect to $\sigma(\mathcal{S}(Y\times Z))$
 and $\mathcal{B}(\mathcal{P}(X))$
and a mapping
$$
(z,x)\mapsto f_z(x), \quad Z\times X\to Y
$$
 that is measurable with respect to $\sigma(\mathcal{S}(Y\times Z))$ and $\mathcal{B}(Y)$.
Then the set
$$
S=\{(y,z)\in Y\times Z\colon\, \nu_z^y (f_z^{-1}(y))=1\}
$$
belongs to $\sigma(\mathcal{S}(Y\times Z))$.
If both  mappings are Borel measurable, then  $S$ is also Borel.
\end{lemma}
\begin{proof}
The pair $(y,z)$ belongs to this set precisely when
$$
\nu_z^y\circ f_z^{-1}=\delta_y .
$$
This is equivalent to the identity
$$
\int_X \psi_j(f_z(x))\, \nu_z^y (dx)=\psi_j(y)
$$
for a fixed countable family $\{\psi_j\}\subset C_b(Y)$ separating
Borel measures on~$Y$ (as recalled above, such  collections exist for all completely regular
Souslin spaces).
Since $\psi_j$  is Borel measurable, it remains to apply  Lemma~\ref{lem2.1} to the space $Y\times Z$.
\end{proof}

We now prove the existence of conditional measures measurably depending on a parameter.
Our proof is a modification of the reasoning used in \cite{Malofeev}, where a somewhat stronger 
assumption was used, but for the reader's convenience
we repeat some steps from \cite{Malofeev} instead of referring to that paper.
Another important reason for this repeating is that we also indicate some changes 
necessary for obtaining conditions for the Borel measurability of conditional measures, which will 
be the subject of the next theorem.

\begin{theorem}\label{t1}
Let
$$
f\colon\, (x,z)\mapsto f_z(x), \quad X\times Z\to Y
$$
be a  Borel mapping. Suppose  that for every $z\in Z$ there is a
Borel probability measure $\mu_z$ on~$X$ such that the mapping
$$
z\mapsto \mu_z, \ Z\to \mathcal{P}(X)
$$
is Borel measurable or, more generally,  $(\sigma(\mathcal{S}(Z)),\mathcal{B}(\mathcal{P}(X)))$-measurable.
Then, for all pairs $(\mu_z, f_z)$,
 there exist proper conditional probabilities $\{\mu^y_z\}_{y\in Y}$ on~$X$
 such that, for every Borel set $B$ in $X$, the function
$$
(y,z)\mapsto \mu_z^y(B)
$$
on $Y\times Z$ is $\sigma(\mathcal{S}(Y\times Z))$-measurable,
i.e., the mapping
$$
(y,z)\mapsto \mu_z^y, \quad Y\times Z\to \mathcal{P}(X)
$$
is measurable with respect to the $\sigma$-algebra $\sigma(\mathcal{S}(Y\times Z))$.
\end{theorem}
\begin{proof}
For every point $z\in Z$, we take the measure
$$
\sigma_z:=\mu_z\circ f_z^{-1}
$$
on $Y$ and an increasing sequence of finite algebras $\mathcal{B}_{z,n}$ the union of which
generates the $\sigma$-algebra
$$
\mathcal{B}_z:=f_z^{-1}(\mathcal{B}(Y)).
$$
Without loss of generality we can assume that $\mathcal{B}_{z,n}$
is generated by some finite partition of $X$ into disjoint sets of the form
$$
A_{z,n,1}=f_z^{-1}(B_{n,1}),\ldots,A_{z,n,m_n}=f_z^{-1}(B_{n,m_n}),
$$
where $B_{n,1},\ldots,B_{n,m_n}$ is a finite partition of $Y$ into disjoint Borel sets
such that the union of $B_{n,i}$ over all $n$ and $i$
generates~$\mathcal{B}(Y)$.
For the space $Y=[0,1)$ one can take $B_{n,i}=[i/n,(i+1)/n)$.
In the general case there is a  continuous injection $T$ of $Y$
into the countable power $[0,1]^\infty$ of $[0,1]$ equipped with the product topology.
Since this is a compact metric space, it can be  covered by finitely many
balls $K_{n,i}$ of radius~$1/n$. So we take  $B_{n,i}=T^{-1}(D_{n,i})$, where $D_{n,1}=K_{n,1}$, 
$D_{n,i+1}=K_{n,i+1}\backslash (K_{n,1}\cup\cdots\cup K_{n,i})$.

The conditional measures for $\mu_z$ and the $\sigma$-algebra $\mathcal{B}_{z,n}$ can be written explicitly:
$$
\mu_{z,n}^y(A)= \sum_{i=1}^{m_n}
\frac{\mu_z(A\cap A_{z,n,i})}{\mu_z(A_{z,n,i})} I_{B_{n,i}}(y),
$$
where $\mu_z(A\cap A_{z,n,i})/\mu_z(A_{z,n,i})=0$ if $\mu_z(A_{z,n,i})=0$.
By Lemma~\ref{lem2.1}
the functions $(y,z)\mapsto \mu_{z,n}^y(A)$ are
$\sigma(\mathcal{S}(Y))\otimes \sigma(\mathcal{S}(Z))$-measurable
(and Borel measurable if $z\mapsto \mu_z$ is Borel), because
$I_{A_{z,n,i}}(x)=I_{B_{n,i}}(f_z(x))$, $I_{A\cap A_{z,n,i}}=I_A I_{A_{z,n,i}}$.
It is obvious  that $\mu_{z,n}^y(A)$ coincides with the conditional expectation
of the function $I_A$ with respect to the $\sigma$-algebra $\mathcal{B}_{z,n}$ and the measure $\mu_z$.
Therefore, for any Borel function $\varphi$
on $X$ the conditional expectation $E_z(\varphi|\mathcal{B}_{z,n})$
of $\varphi$ with respect to $\mathcal{B}_{z,n}$ and $\mu_z$ equals
$$
\int_X \varphi(x)\, \mu_{z,n}^y(dx).
$$
According to the martingale convergence theorem (see \cite[Section~10.3]{B07}, for every fixed~$z$,  the constructed functions
$E_z(\varphi|\mathcal{B}_{z,n})$ converge $\sigma_z$-almost everywhere
and in the space $L^1(\sigma_z)$ to the conditional expectation $E_z(\varphi|\mathcal{B}_z)$
of $\varphi$ with respect to the $\sigma$-algebra
$\mathcal{B}_z$ and the measure $\mu_z$.

However, we need conditional measures, not conditional expectations. Of course,
it is known that some conditional measures $\mu_z^y$ exist and define
the same conditional expectations.
Unfortunately, not every choice of $\mu_z^y$ is suitable to guarantee the joint measurability in~$(y,z)$, because the relations defining
conditional expectations hold almost everywhere, not pointwise,
 and the corresponding measure zero sets depend on~$\varphi$ and~$z$. So a constructive method
of selecting conditional measures is needed.

In order to define our conditional probabilities, we  consider the set of points for which
the sequence of measures $\mu_{z,n}^y$ converges
and its limit is concentrated on the set $f_z^{-1}(y)$. Convergence is easier achieved on a compact space.
By using a countable family in $C_b(X)$
separating points, we can embed $X$ continuously into the cube $I:=[0,1]^\infty$ and
assume that $X$ is a Souslin set in $I$ (equipped with a stronger
topology than the one induced from $[0,1]^\infty$).
The countable family of polynomials in coordinate functions  of the form
$\sum c_{i_1,\ldots,i_m,k_1,\ldots,k_m}x_{i_1}^{k_1}\cdots x_{i_m}^{k_m}$, where $c_{i_1,\ldots,k_m}$ are
rational numbers, will be denoted by~$\{\varphi_j\}$.

We denote by $\Omega$ the set of all points $(y,z)\in Y\times Z$ for which,  for every $\varphi_j$,
the sequence of integrals
$$
\int_X \varphi_j(x)\, \mu_{z,n}^y(dx)
$$
has a finite limit as $n\to\infty$.
Every integral is a $\sigma(\mathcal{S}(Y\times Z))$-measurable function of $(y,z)$. Moreover, it is Borel if
the mapping $z\mapsto \mu_z$ is Borel measurable.
Hence $\Omega\in \sigma(\mathcal{S}(Y\times Z))$ and $\Omega\in\mathcal{B}(Y\times Z)$ if
$z\mapsto \mu_z$ is Borel measurable.

We now use the compactness of $I$, due to which for any $(y,z)\in \Omega$
the sequence of measures $\mu_{z,n}^y$ regarded on~$X$ converges weakly to a Borel probability measure
$\nu_z^y$ on~$I$ (but so far we do not claim that it is concentrated on~$X$).

According to Lemma~\ref{lem2.2}, the subset
$$
\Omega_0:=\{(y,z)\in \Omega\colon \nu_z^y(f_z^{-1}(y))=1\}
$$
belongs to $\sigma(\mathcal{S}(Y\times Z))$ (and is  Borel
if $z\mapsto \mu_z$ is Borel).
For each $(y,z)\in \Omega_0$, the measure $\nu_z^y$ is obviously concentrated
on~$X$ (recall that the Souslin set $X$ is measurable with respect to all Borel measures).
The set
$$
\Omega_1=\{(y,z)\in Y\times Z\colon\ y\in f_z(X)\}
$$
is the projection of the graph of~$f$ under the mapping
$$
X\times Z\times Y\to Y\times Z,\quad (x,z,y)\mapsto (y,z).
$$
This set is Souslin in~$Y\times Z$.
If each $f_z$ is a surjection, then $\Omega_1=Y\times Z$.

Now we are going to apply the measurable choice theorem (see \cite[Theorem~6.9.2]{B07})
to the multivalued mapping $\Psi\colon\, (y,z)\mapsto f_z^{-1}(y)$ on $\Omega_1$
with values in the class of non-empty subsets of~$X$.
Its graph is the set
$$
\{(y,z,u)\colon\, (y,z)\in \Omega_1, \ u\in f_z^{-1}(y)\}
=\{(y,z,u)\colon\, (y,z)\in \Omega_1, \ f_z(u)=y\} ,
$$
which is Souslin, because $(y,z,u)\mapsto f_z(u)$ and
$(y,z,u)\mapsto y$ are Borel mappings.
By the cited theorem there is a mapping
$$
g\colon\, (y,z)\mapsto g_z(y), \   \Omega_1\to X
$$
such that
$$
g_z(y)\in f_z^{-1}(y) \quad \forall\,  z\in Z,  y\in f_z(X)
$$
and $g$ is measurable with respect to the
restriction of $\sigma(\mathcal{S}(Y))\otimes\sigma(\mathcal{S}(Z))$ to~$\Omega_1$
and~$\mathcal{B}(X)$.

Finally, if $(y,z)\not\in \Omega_0$ and $y\in f_z(X)$, i.e., $(y,z)\in \Omega_1$, we set $\nu_z^y:=\delta_{g_z(y)}$,
and if  $y\not\in f_z(X)$, we set $\nu_z^y:=\delta_{x_0}$, where $x_0\in X$ is a fixed point
independent of $y$ and~$z$. The constructed family of measures $\nu_z^y$ is measurable
with respect to $\mathcal{S}(Y\times Z)$.
Indeed,  its restriction to $\Omega_0$ is measurable with respect to the trace of $\sigma(\mathcal{S}(Y))\otimes \sigma(\mathcal{S}(Z))$.
The restriction to $\Omega_1\backslash \Omega_0$ is measurable with respect to the trace of
$\sigma(\mathcal{S}(Y\times Z))$, because for every $f\in C_b(X)$  the function $f(g_z(y))$
is $\sigma(\mathcal{S}(Y\times Z))$-measurable by the measurability of~$g$.
The restriction to the complement of $\Omega_0\cup \Omega_1$ is constant, and both sets
$\Omega_0$ and $\Omega_1$ are in $\sigma(\mathcal{S}(Y\times Z))$.

It remains to verify that the  measures $\nu_z^y$ serve as
conditional probabilities with the required properties.
By definition $\nu_z^y(f_z^{-1}(y))=1$ for all $y\in Y$ and $z\in Z$. The function
$$
(y,z)\mapsto
\int_X \varphi_j(x)\, \nu_z^y(dx)
$$
is $\sigma(\mathcal{S}(Y\times Z))$-measurable for every $\varphi_j$, which gives the required measurability of~$\{\nu_z^y\}$.
In order to see that condition~3) from the definition of proper regular conditional measures holds, it is enough
to show that,
picking arbitrary regular conditional measures $\mu_z^y$ for $\mu_z$
(not necessarily measurable  in~$z$), for every fixed $z\in Z$, we have
$$
\nu_z^y=\mu_z^y \quad \hbox{for $\sigma_z$-almost every $y$.}
$$
By the definition of $\sigma_z$ this is equivalent to the relation
$$
\nu_z^{f(x)}=\mu_z^{f(x)} \quad \hbox{for $\mu_z$-almost every $x$.}
$$
This relation holds, since there is a countable family of bounded continuous
functions on $X$ separating Borel measures, and for
every function $\psi$ from this family its integrals
against $\nu_z^{f(x)}$ and $\mu_z^{f(x)}$ coincide $\mu_z$-almost everywhere, because, as explained above,  both expressions
$$
\int_X \psi(u) \, \mu_z^{f(x)}(du)\quad \hbox{and}\quad \int_X \psi(u) \, \nu_z^{f(x)}(du)
$$
serve as the conditional expectation of $\psi$ with respect to $\mathcal{B}_z$ and~$\mu_z$.
\end{proof}

\begin{theorem}\label{t2}
Suppose that in Theorem~{\rm\ref{t1}}
for each $z$ the mapping $f_z\colon X\to Y$
is a Borel surjection possessing a right inverse mapping~$g_z$ such that
$(y,z)\mapsto g_z(y)$ is Borel measurable
{\rm(}or, more generally, the set $\bigcup_z (f_z(X)\times \{z\})$ is Borel in $Y\times Z$
and the mapping $(y,z)\mapsto g_z(y)$ is Borel measurable{\rm)}; for example,
the mapping $f\colon X\to Y$ does not
depend on~$z$ and is a Borel surjection possessing a Borel right inverse mapping~$g$.
If also $z\mapsto\mu_z$ is Borel measurable,
then there exists a jointly Borel measurable version of conditional measures~$\mu^y_z$.

In particular, this is true if $X$ is the product of two Souslin
spaces~$X_1$ and~$X_2$, $f$~is the standard projection onto~$X_2$,
and $z\mapsto \mu^z$ is Borel measurable.
\end{theorem}
\begin{proof}
This follows from our reasoning above (taking into account the notes about
Borel measurability), since under stronger assumptions of this theorem
we already have jointly
Borel measurable right inverse mappings $g_z$ without any measurable choice theorems.
The corresponding Dirac measures $\delta_{g_z(y)}$ defined
for all $y,z$ from the complement of the Borel set $\Omega_0$
are also jointly Borel measurable
and $\Omega_1=Y\times Z$ in the surjective case. Similarly we consider
the case where $f_z$ is not surjective
and $\bigcup_z (f_z(X)\times \{z\})$ is Borel.
\end{proof}

Note that in the case of the product-space $X=X_1\times X_2$ and the projection $\pi_{X_2}$ on~$X_2$
it is sometimes more convenient to consider conditional measures on the common space $X_1$
in place of the slices $X_1\times \{x_2\}=\pi_{X_2}^{-1}(x_2)\subset X_1\times X_2$.
Both representations are equivalent and it is easy to pass from one to the other.

The assertion with a single mapping $f$ not depending on the parameter admits an obvious
generalization.

\begin{corollary}
Let $X_1$ and $X_2$ be completely regular Souslin spaces, let $(T,\mathcal{T})$ be a measurable
space, and let $t\mapsto \mu_t$ be a mapping from $T$ to $\mathcal{P}(X_1\times X_2)$ that is
measurable with respect to $\mathcal{T}$ and $\mathcal{B}(\mathcal{P}(X_1\times X_2))$.
Then there is a mapping $$(t,x_2)\mapsto \mu_t^{x_2}\in \mathcal{P}(X_1),$$
measurable with respect to  $\mathcal{T}\otimes\mathcal{B}(X_2)$ and
$\mathcal{B}(\mathcal{P}(X_1))$, such that the measures $\mu_t^{x_2}$ serve
as conditional measures for $\mu_t$ and the projection on~$X_2$.
\end{corollary}
\begin{proof}
The previous theorem can be applied with the space $Z=\mathcal{P}(X_1\times X_2)$
as a parameter space, which gives a Borel mapping $(\mu,x_2)\mapsto P^{x_2}_\mu$
such that $P^{x_2}_\mu$ serve as conditional measures for~$\mu$. Then the mapping
$(t,x_2)\mapsto \mu_t^{x_2}:=P_{\mu_t}^{x_2}$ is
measurable with respect to  $\mathcal{T}\otimes\mathcal{B}(X_2)$.
\end{proof}

The following parametric version of the so-called gluing lemma
(see~\cite{V}) has
been noted in \cite[Theorem~7.3]{KNS} (for Polish spaces).

\begin{corollary}
Let $X_1, X_2, X_3$ be completely regular Souslin spaces,
let $(T,\mathcal{T})$ be a measurable
space, and let
$$
t\mapsto \mu_{1,2,t}, \
T\to \mathcal{P}(X_1\times X_2)
\quad\hbox{and}\quad
t\mapsto \mu_{2,3,t}, \ T\to \mathcal{P}(X_2\times X_3)
$$
be $\mathcal{T}$-measurable mappings such that,
for each~$t$, the projections of $\mu_{1,2,t}$ and $\mu_{2,3,t}$ on~$X_2$ coincide.
Then there is a $\mathcal{T}$-measurable mapping $t\mapsto \eta_t$ from $T$
to the space $\mathcal{P}(X_1\times X_2\times X_3)$ such that, for each~$t$, the projection
of $\eta_t$ on $X_1\times X_2$ is $\mu_{1,2,t}$ and the projection on~$X_2\times X_3$ is~$\mu_{2,3,t}$.
\end{corollary}
\begin{proof}
It suffices to recall the usual construction of the measure on $X_1\times X_2\times X_3$
for every fixed~$t$
via conditional measures (see \cite[Lemma~3.3.1]{B18} or~\cite{V03}):
using disintegrations
$$
\mu_{1,2,t}(dx_1 dx_2)=\mu_{1,2,t}^{x_2}(dx_1)\pi_t(dx_2),
\quad
\mu_{2,3,t}(dx_2 dx_3)=\mu_{2,3,t}^{x_2}(dx_3)\pi_t(dx_2),
$$
where $\pi_t$ is the common projection of $\mu_{1,2,t}$ and $\mu_{2,3,t}$ on~$X_2$,
$\mu_{1,2,t}^{x_2}$ and $\mu_{2,3,t}^{x_2}$ are the
 corresponding conditional measures measurably depending on~$t$,
we set
$$
\eta_t(dx_1 dx_2 dx_3)=\mu_{1,2,t}^{x_2}(dx_1)\mu_{2,3,t}^{x_2}(dx_3)\pi_t(dx_2).
$$
This means that for each bounded Borel function $f$ on $X_1\times X_2\times X_3$
we have the  following equality:
$$
\int f\, d\eta_t=\int_{X_2}\!\int_{X_3}\!\int_{X_1} f(x_1,x_2,x_3)\,
\mu_{1,2,t}^{x_2}(dx_1)
\, \mu_{2,3,t}^{x_2}(dx_3)\, \pi_t(dx_2).
$$
The measurability of the mapping $t\mapsto \eta_t$ follows by the measurability of conditional
measures and the projection. The fact that $\eta_t$ has the prescribed projections
is verified directly (see
\cite[Lemma~3.3.1]{B18}).
\end{proof}

\begin{remark}
\rm
The assumption that the space of parameters $Z$ is Souslin is quite natural
in the situation of Theorem~\ref{t1}. However, in the situation of Theorem~\ref{t2}
for $Z$ we can take an arbitrary measurable space $(Z,\mathcal{Z})$
without any topology. The same reasoning shows that if $(z,x)\mapsto f_z(x)$ is
$\mathcal{Z}\otimes \mathcal{B}(X)$-measurable, each $f_z$ is a surjection that admits
a right inverse mapping $g_z$ for which $(z,y)\mapsto g_z(y)$ is
$\mathcal{Z}\otimes \mathcal{B}(Y)$-measurable, and $\mu_z$ is $\mathcal{Z}$-measurable,
then there are conditional measures $\mu_z^y$, measurable with respect to
$\mathcal{Z}\otimes \mathcal{B}(Y)$.
\end{remark}

\begin{remark}
\rm
(i)
It is known  that the $\sigma$-algebra $\sigma(\mathcal{S}(X))$
is not countably generated for  any uncountable Polish space~$X$
(see \cite[Example~6.5.9]{B07}), unlike the Borel $\sigma$-algebra.
This is one of the reasons why the Borel measurability can be preferable in applications.

(ii) One can show that if $Y$ and $Z$ are uncountable Souslin spaces, then
the product $\sigma$-algebra $\sigma(\mathcal{S}(Y)\otimes \mathcal{S}(Z))$ is strictly smaller
than the $\sigma$-algebra $\sigma(\mathcal{S}(Y\times Z))$ of the product-space.

(iii) The existence of conditional expectations
measurable with respect to a parameter can be obtained
under broader assumptions, because
in this case there is no problem with property 1) of conditional measures.
The continuity  of conditional expectations with respect to a parameter was
studied in~\cite{HP}.
\end{remark}

We now see that the existence of jointly Borel conditional measures depending on
the parameter $z$ implies some restrictions on the mappings~$f_z$, so that such
joint Borel measurability cannot be always guaranteed.
The next proposition is a parametrized version of the known result
of Blackwell and Ryll-Nardzewski \cite{BR} for single measures.

\begin{proposition}\label{p1}
Let $X,Y,Z$ be Polish spaces.
Suppose that there is a jointly Borel measurable version
of conditional measures $\mu^y_z$ concentrated on the sets $f_z^{-1}(y)$ for all $y\in Y$
and $z\in Z$. Then there is a Borel mapping $g\colon Z\times Y\to X$ such that
$f_z(g(z,y))=y$ for all $y\in Y$ and $z\in Z$.
\end{proposition}
\begin{proof}
We shall use the following result of
Blackwell and Ryll-Nardzewski \cite{BR}
(see also \cite[Corollary~18.7]{Kech} or
\cite[Exercise~10.10.47]{B07}, where the hint contains the proof).
For our convenience we change their notation of spaces.
Suppose that $U$ and $X$ are Borel sets in
Polish spaces, $\mathcal{A}$ is a countably generated sub-$\sigma$-algebra
in $\mathcal{B}(U)$ and for each $u\in U$ there is a measure $\mu^u\in \mathcal{P}(X)$ such that
the function $u\mapsto \mu^u(B)$ is $\mathcal{A}$-measurable for every set $B\in \mathcal{B}(X)$.
Let $S\subset X\times U$ be a set such that $\mu^u(S_u)>0$ for all $u\in U$,
where $S_u=\{x\in X\colon (x,u)\in S\}$. Then $S$ contains the graph
of an $(\mathcal{A},\mathcal{B}(X))$-measurable mapping $F\colon U\to X$.

We apply this result in the situation where $U=Z\times X$,
$\mathcal{A}$ is the sub-$\sigma$-algebra in $\mathcal{B}(Z\times X)$ generated by
the mapping
$$
h\colon Z\times X\to Z\times Y, \quad (z,x)\mapsto (z, f_z(x)),
$$
$$
\mu^u=\mu^{f_z(x)}_z, \quad u=(z,x),
$$
 and
$$
S=\{(z,x,v)\in Z\times X\times X\colon f_z(v)=f_z(x)\}.
$$
The section $S_u$ is defined by
$$
S_u=S_{z,x}=\{v\in X\colon f_z(v)=f_z(x)\}=f_z^{-1}(f_z(x)),
$$
hence $\mu^u(S_u)=\mu^{f_z(x)}_z(f_z^{-1}(f_z(x)))=1$. By the cited result there is
a mapping $F\colon Z\times X\to X$ with the graph in $S$ such that $F$ is $\mathcal{A}$-measurable.
The latter means that there is a Borel mapping $g\colon Z\times Y\to X$ such that
$F(z,x)=g(h(z,x))$. Since the graph of $F$ is contained in~$S$, by the definition of $h$ we obtain
$$
f_z(g(z,f_z(x)))=f_z(x) \quad \forall x\in X, z\in Z.
$$
It follows that $f_z(g(z,y))=y$ for all $z\in Z$ and $y\in Y$.
\end{proof}

It is known that in general there is no $g$ with the stated properties
(see, e.g., \cite[\S6.9]{B07}). A~sufficient condition
for the existence of $g$ is this: for each $y\in Y$ and $z\in Z$ the set $f_z^{-1}(y)$
is a countable union of compact sets. Indeed, we consider again the Borel mapping
$h\colon (z,x)\mapsto (z,f_z(x))$ and observe that the sets $h^{-1}(z,y)$ are countable unions
of compact sets. Hence by a classical result (see Theorem~C in the next section)
there is a Borel mapping $g\colon Z\times Y\to X$ with the graph
contained in the set $\{(z,y,x)\colon f_z(x)=y\}$.

\section{Kantorovich problems with a parameter}

We now turn to Kantorovich optimal plans depending on a parameter.

Let $X$ and $Y$ be completely regular Souslin spaces
(for example, Polish spaces). The corresponding spaces of Borel
probability  measures $\mathcal{P}(X)$ and $\mathcal{P}(Y)$
will be equipped with their
weak topologies (making them Souslin or Polish spaces, respectively).
By $\pi_X$ and $\pi_Y$ we denote the projections of $X\times Y$ on $X$ and~$Y$.

For any pair of measures $\mu\in \mathcal{P}(X)$ and
$\nu\in \mathcal{P}(Y)$, the set
$$
\Pi(\mu,\nu)=\{\sigma\in \mathcal{P}(X\times Y)\colon
\sigma\circ \pi_X^{-1}=\mu, \sigma\circ \pi_Y^{-1}=\nu\}
$$
is  convex and compact in the weak topology, which follows from Prohorov's theorem.
This set is not empty: it always contains the product of $\mu$ and $\nu$.

Recall that a function $f$ is lower semicontinuous if the sets $\{f\le c\}$ are closed.
It is known (see \cite[Corollary~4.3.4]{B18}) that
if $f$ is a bounded lower semicontinuous function on $X$ and Borel
probability measures $\mu_n$ on~$X$ converge weakly to~$\mu$, then
$$
\int_X f\, d\mu\le \liminf_{n\to\infty} \int_X f\, d\mu_n.
$$

Given a lower semicontinuous cost function $h\ge 0$ on $X\times Y$ and
a pair of measures $\mu\in \mathcal{P}(X)$ and
$\nu\in \mathcal{P}(Y)$,
in the aforementioned Kantorovich problem of finding
the infimum of $K_h(\mu,\nu)$ of the quantity
$I_h(\sigma)$
over all measures $h\in \Pi(\mu,\nu)$ the minimum
is attained if there is a measure
$\sigma$ with $I_h(\sigma)<\infty$ (which is always true if $h$ is bounded).

Let $(T,\mathcal{T})$ be a measurable space. In the case where $T$ is a topological space
we assume that $\mathcal{T}$ is its Borel $\sigma$-algebra $\mathcal{B}(T)$.

Assume also that
 $$
 h\colon T\times X\times Y\to [0,+\infty)
 $$
is a $\mathcal{T}\otimes\mathcal{B}(X)\otimes\mathcal{B}(Y)$-measurable
function such that $h_t\colon (x,y)\mapsto h(t,x,y)$ is lower
semicontinuous for  each~$t$.

Thus, we obtain a Kantorovich problem with a parameter. Dependence on a parameter
appears even for a single cost function if marginal distributions depend on~$t$.
We consider the case where both marginals and the cost function depend on $t$.

Let
$t\mapsto \mu_t$, $T\to \mathcal{P}(X)$  be a
$(\mathcal{T},\mathcal{B}(\mathcal{P}(X)))$-measurable mapping and let
$t\mapsto \nu_t$, $T\to \mathcal{P}(Y)$  be a
$(\mathcal{T},\mathcal{B}(\mathcal{P}(Y)))$-measurable mapping.

\begin{theorem}\label{tmain1}
Suppose that
the transportation costs $K(t):=K_{h_t}(\mu_t,\nu_t)$ are finite
and the cost functions $h_t\colon (x,y)\mapsto h(t,x,y)$ are continuous.
Then the function
$K$  is $(\mathcal{T}, \mathcal{B}(\mathcal{P}(X\times Y)))$-measurable.
In addition, one can choose optimal measures $\sigma_t$
such that the mapping $t\mapsto \sigma_t$ is measurable with respect
to $\sigma(\mathcal{S}(\mathcal{T}))$ and $\mathcal{B}(\mathcal{P}(X\times Y))$.
\end{theorem}

In the next theorem we remove the  assumption of continuity of
 cost functions and reinforce the conclusion by the existence
 of Borel measurable selections, but $T$ is required to be a Souslin space.

\begin{theorem}\label{tmain2}
Suppose that $T$ is a Souslin space,
$t\mapsto \mu_t$ and $t\mapsto \nu_t$ are Borel mappings with values
in the spaces $\mathcal{P}(X)$ and $\mathcal{P}(Y)$, respectively,
and the corresponding transportation costs $K_{h_t}(\mu_t,\nu_t)$ are finite.
Then
the function $t\mapsto K_{h_t}(\mu_t,\nu_t)$ is Borel measurable and
there is a mapping $t\mapsto \sigma_t$, $T\to \mathcal{P}(X\times Y)$,
measurable with respect to
$\mathcal{B}(T)$ and $\mathcal{B}(\mathcal{P}(X\times Y))$, such that
for all $t\in T$ we have
$$
\sigma_t\in \Pi(\mu_t,\nu_t), \quad \int h(t,x,y)\, \sigma_t(dx dy)=K_{h_t}(\mu_t,\nu_t).
$$
\end{theorem}

\begin{corollary}
In the previous theorem, there is a sequence of Borel measurable mappings
$\Phi_n\colon T\to \mathcal{P}(X\times Y)$ such that, for every $t\in T$,
 the sequence $\{\Phi_n(t)\}$ is dense in the convex compact set $M_t$ of $h_t$-optimal measures
 in $\Pi(\mu_t,\nu_t)$.
\end{corollary}

\begin{corollary}
In the previous theorem, the optimal plans $\sigma_t$ admit disintegrations
$$
\sigma_t=\int_Y \sigma_t^y \, \nu_t(dy)
$$
 with Borel probability measures $\sigma_t^y$ on $X$ that are Borel measurable in~$(t,y)$.
\end{corollary}

For Souslin spaces $X$ and $Y$ we have the following result.

\begin{theorem}\label{tmain3}
Let $X$ and $Y$ be completely regular Souslin spaces and let $T$ be a Souslin space.
Let $(x,y)\mapsto h(t,x,y)$ be continuous for every~$t$ and
let $t\mapsto\mu_t$ and $t\mapsto\nu_t$ be Borel measurable. Then the function $t\mapsto K(t)$
is measurable with respect to~$\sigma(\mathcal{S}(T))$.
\end{theorem}

Note that if in the last theorem
the function $t\mapsto K(t)$ is Borel measurable, then there is a sequence
of mappings $\Phi_n\colon T\to \mathcal{P}(X\times Y)$,
measurable with respect to $(\sigma(\mathcal{S}(T)),\mathcal{B}(\mathcal{P}(X\times Y)))$,
 such that, for every $t\in T$,
 the sequence $\{\Phi_n(t)\}$ is dense in the convex compact set $M_t$ of $h_t$-optimal measures
 in $\Pi(\mu_t,\nu_t)$.

\section{Auxiliary results and proofs}

The following general version of the Kantorovich duality for finite nonnegative
lower semicontinuous cost functions holds:
\begin{multline}\label{gen-dual}
K_h(\mu,\nu)=\sup \biggl\{\int \varphi\, d\mu + \int \psi\, d\nu
\colon\\
 \varphi\in C_b(X), \psi\in C_b(Y), \varphi(x)+\psi(y)\le h(x,y)\biggr\}.
\end{multline}
See   \cite{BeigLS14}, \cite{BeigS}, \cite{Kellerer84}, \cite{RR}, and~\cite{V}
for a discussion of this duality; a short derivation of the general case from the case
of bounded continuous cost functions is given in~\cite{BeigS}.
Hence for each $\varepsilon>0$ there are
functions
$\varphi\in C_b(X)$ and $\psi\in C_b(Y)$ such that
$$
\varphi(x)+\psi(y)\le h(x,y) \quad \hbox{for all $x$ and $y$}
$$
 and
$$
K_h(\mu,\nu)\le \int\varphi\, d\mu+\int\psi \, d\nu+\varepsilon.
$$
Moreover, for bounded~$h$, in the right-hand side of (\ref{gen-dual}) one can take the supremum over
 $\varphi$ and $\psi$ such that $|\varphi|\le \|h\|_\infty$,
$|\psi|\le \|h\|_\infty$. This is explained in \cite[Remark~1.13]{V}, but for the
reader's convenience we give a straightforward justification.
We can assume that $\|h\|_\infty=1$.
If a pair $\varphi, \psi$ satisfies the indicated bound,
then, for any number~$t$, the pair $\varphi +t, \psi -t$ also satisfies
this bound and the sum of the corresponding
integrals is the same. Hence we can assume that $\sup_x \varphi(x)=1$.
Hence $\psi(y)\le 0$.
Next, we replace $\varphi$ by $\varphi_1=\max(\varphi, 0)$
and obtain a pair with $\varphi_1(x)+\psi(y)\le h(x,y)$
and $0\le \varphi_1\le 1$ for which
the integral of $\varphi_1$ is not less than that of~$\varphi$.
Finally,  we replace $\psi$ by $\psi_1=\max(\psi, -1)$,
which keeps the upper bound by $h$ and increases the integral.
Hence we obtain a pair with $0\le\varphi_1\le 1$, $-1\le \psi_1\le 0$
and $\varphi_1(x)+\psi_1(y)\le h(x,y)$ for which the sum of the respective
 integrals dominates the original sum.
The next lemma is an immediate corollary of this bound.

\begin{lemma}
Let $h\le 1$. Then for all $\mu_1,\mu_2\in \mathcal{P}(X)$
and $\nu_1,\nu_2\in \mathcal{P}(Y)$ we have
$$
|K_h(\mu_1,\nu_1)-K_h(\mu_2,\nu_2)|\le \|\mu_1-\mu_2\|+\|\nu_1-\nu_2\|.
$$
\end{lemma}
\begin{proof}
We can assume that $K_h(\mu_1,\nu_1)>K_h(\mu_2,\nu_2)$.
Let $\varepsilon>0$. There are functions $\varphi\in C_b(X)$, $\psi\in C_b(Y)$  with
$\varphi(x)+\psi(y)\le h(x,y)$,
$|\varphi|\le 1$, $|\psi|\le 1$ such that
$$
K_h(\mu_1,\nu_1)<\int\varphi\, d\mu_1+\int\psi \, d\nu_1+\varepsilon.
$$
Since
$$
K_h(\mu_2,\nu_2)\ge \int\varphi\, d\mu_2+\int\psi \, d\nu_2,
$$
we have
$$
K_h(\mu_1,\nu_1)-K_h(\mu_2,\nu_2)\le \varepsilon +
\int\varphi\, d(\mu_1-\mu_2)+\int\psi \, d(\nu_1-\nu_2),
$$
whence our claim follows with the extra term $\varepsilon$ on the right, so it remains
to let $\varepsilon\to 0$.
\end{proof}

\begin{lemma}\label{lem1}
Let $(T,\mathcal{T})$ be a measurable space, $Z$ a Polish space,
 and let $t\mapsto \mu_t$ be a mapping from $T$ to $\mathcal{P}(Z)$ measurable with respect
 to $\mathcal{T}$ and $\mathcal{B}(\mathcal{P}(Z))$. Then there is a sequence of
 increasing compact sets~$Z_n(t)\subset Z$ such that the sets $\bigcup_t (\{t\}\times Z_n(t))$
 are in $\mathcal{T}\otimes\mathcal{B}(Z)$, the set-valued mapping $t\mapsto Z_n(t)$ is $\mathcal{T}$-measurable,
 the normalized
 restrictions $\mu_t^n$ of $\mu_t$ to $Z_n(t)$ define
   mappings $t\mapsto \mu_t^n$ from $T$ to $\mathcal{P}(Z)$
 measurable in the same sense and $\|\mu_t^n-\mu_t\|\to 0$.

 The same is true if $Z$ is a completely regular Luzin space, hence this is true
 if $Z$ is a Borel set in a Polish space.
 \end{lemma}
\begin{proof}
It suffices to introduce the parameter $t$ in the standard proof of Ulam's theorem.
We consider $Z$ with a complete separable metric and take a dense countable set $\{z_j\}\subset Z$.
Let $n\in\mathbb{N}$.
For each $k$ and $m$ in $\mathbb{N}$
 let $A_{k,m}$ be the union of $m$ closed balls of radius $2^{-k}$ centered
at $z_1,\ldots,z_m$. Then $\mu_t(A_{k,m})\to 1$ as $m\to\infty$.
Let
$$
N_{n,k}(t)=\min \{m\colon \mu_t(A_{k,m})>1-2^{-n-k}\},
$$
$$
Z_{n}(t)=\bigcap_{k\ge 1} A_{k, N_{n,k}}(t).
$$
The sets $A_{k, N_{n,k}}(t)$ are closed. Hence the sets $Z_{n}(t)$ are also closed.
In addition, each $Z_{n}(t)$ is contained in finitely many balls of radius $2^{-k}$ for each~$k$.
Hence $Z_{n}(t)$ is compact. By construction,
$$
\mu_t(X\backslash Z_{n}(t))< \sum_{k=1}^\infty 2^{-n-k}=2^{-n}.
$$
Let $\mu_t^n$ be the normalized restriction of $\mu_t$ to $Z_{n}(t)$.
Then $\|\mu_t-\mu_t^n\|< 2^{-n}(1-2^{-n})^{-1}$.
We have $Z_n(t)\subset Z_{n+1}(t)$, since $N_{n,k}(t)\le N_{n,k+1}(t)$, so
$A_{k, N_{n,k}}(t)\subset A_{k, N_{n,k+1}}(t)$.

The functions $t\mapsto N_{n,k}(t)$ are $\mathcal{T}$-measurable,
since the set $N_{n,k}^{-1}(q)$ is the intersection
of the sets $\{t\colon \mu_t(A_{k,j})\le 1-2^{-n-k}\}$ with $j<q$
and $\{t\colon \mu_t(A_{k,q})> 1-2^{-n-k}\}$ that are $\mathcal{T}$-measurable, which
readily follows from the measurability of $t\mapsto \mu_t$.
In order to show the measurability of $\mu_t^n$ it suffices to show the measurability
of the mapping $t\mapsto \mu|_{N_{n}(t)}$. This mapping is the limit of restrictions
of $\mu_t$ to the decreasing sets
$\bigcap_{k=1}^m  A_{k, N_{n,k}}(t)$. Such restrictions
$\mu_t^{n,m}$ are $\mathcal{T}$-measurable. Indeed,
the sets $N_{n,k}^{-1}(q)$ are $\mathcal{T}$-measurable, hence so are their finite
intersections, but $\mu_t^{n,m}$ has countably many values assumed  on such intersections.

Every set $\bigcup_t (\{t\}\times Z_n(t))$
 belongs to $\mathcal{T}\otimes \mathcal{B}(Z)$, because it is the intersection of the sets
 $\bigcup_t \Bigl(\{t\}\times \bigcap_{k=1}^m  A_{k, N_{n,k}}(t)\Bigr)$, which are
 in $\mathcal{T}\otimes \mathcal{B}(Z)$,
 since they are countable unions of sets of the form $T_{k,n,m}\times A_{k,m}$ with
 $T_{k,n,m}=\{t\colon N_{k,n}(t)=m\}$.
Let us show that the set-valued mapping $t\mapsto Z_n(t)$ is $\mathcal{T}$-measurable.
It suffices to show that for every $x\in X$ the real function $t\mapsto {\rm dist}(x,Z_n(t))$
is $\mathcal{T}$-measurable, see \cite[Theorem~III.9]{CV} or \cite[Chapter~8]{AubinF}.
Let $D_{n,m}(t)=\bigcap_{k=1}^m A_{k, N_{n,k}}(t)$.
We observe that
$$
d_H(Z_n(t), D_{n,m}(t))\to 0 \quad  \hbox{and}\quad
{\rm dist}(x,D_{n,m}(t))\to {\rm dist}(x,Z_n(t)) \quad \hbox{as $m\to\infty$.}
$$
Indeed, for every fixed $\varepsilon>0$ there is $m$ such that $D_{n,m}(t)$ is contained
in the $\varepsilon$-neighborhood of~$Z_n(t)$, because otherwise there is a sequence
of points $x_m\in D_{n,m}(t)$ with ${\rm dist}(x_m,Z_n(t))\ge \varepsilon$. Each $D_{n,m}(t)$
is a union of finitely many balls of radius $2^{-k}$, hence $\{x_m\}$ is precompact and has a limit
point~$x_0$. This point must belong to all $D_{n,m}(t)$, hence to~$Z_n(t)$, which is impossible,
since ${\rm dist}(x_0,Z_n(t))\ge \varepsilon$. This proves the first relation, the second is its corollary.

The case of Luzin spaces follows from the considered case, because  $Z$ admits a stronger
Polish topology that generates a stronger Polish topology
on $\mathcal{P}(Z)$ with the same Borel sets as in the original topology,
so the measurability of $\mathcal{P}(Z)$-valued mappings remains the same.
Finally, we recall that any Borel set in a Polish space is the image
of a Polish space under a continuous injective mapping (see \cite[Corollary~6.8.5]{B07}).
\end{proof}

\begin{remark}
\rm
Under a stronger condition that $t\mapsto \mu_t(A)$ is $\mathcal{T}$-measurable
for every Souslin set~$A$ (which does not follow automatically) the previous
assertion extends to the case of a Souslin subspace $Z$ in a Polish space~$E$
and gives increasing compact sets $Z_n(t)$ such that the functions $(t,x)\mapsto I_{Z_n(t)}(x)$
are $\mathcal{T}\otimes\sigma(\mathcal{S}(Z))$-measurable and
$\mu_t(Z_n(t))> 1-2^{-n}$. To this end, we first take such compact
sets $Z_n^1(t)$ in~$E$ and then consider a parametric version of the standard proof
of measurability of sets obtained by means of the Souslin operation
(see \cite[Theorem~1.10.5]{B07}). Recall that $Z$ can be written as
$$
Z=\bigcup_{(n_i)} \bigcap_{k=1}^\infty E_{n_1,\ldots,n_k},
$$
where $\{E_{n_1,\ldots,n_k}\}$ is a certain monotone table of
 closed balls of rational radii centered at points
of a fixed countable dense set and the union is taken over all natural sequences~$(n_i)$.
For every
collection $m_1,\ldots,m_k$
of natural numbers, denote by $D_{m_1,\ldots,m_k}$
the union of the sets $E_{n_1,\ldots,n_k}$ over all $n_1\le m_1,\ldots,n_k\le m_k$.
This is a closed set. It is clear from the proof of the cited theorem
(taking into account Remark~\ref{rem3.2})
that
one can find numbers $m_k(t)$ measurably depending on~$t$ such that
$$
\mu_t(D_{m_1(t),\ldots,m_k(t)}\cap Z_n^1(t))> 1-2^{-n}.
$$
 Then
$$
\mu_t\Bigl(\bigcap_{k=1}^\infty D_{m_1(t),\ldots,m_k(t)}\cap Z_n^1(t)\Bigr)\ge 1-2^{-n}.
$$
It is verified in that proof that
$\bigcap_{k=1}^\infty D_{m_1(t),\ldots,m_k(t)}$ is contained in~$Z$. It is clear that
this set is closed, so its intersection with $Z_n^1(t)$ is compact.
\end{remark}

\begin{lemma}\label{lem2}
Suppose that
$(T,\mathcal{T})$ is a measurable space,
$X$ and $Y$ are Polish {\rm(}or Luzin{\rm)} spaces,
$t\mapsto \mu_t$ and $t\mapsto \nu_t$ are $\mathcal{T}$-measurable
mappings with values in $\mathcal{P}(X)$ and $\mathcal{P}(Y)$, correspondingly.
Let $(x,y)\mapsto h(t,x,y)$ be lower
semicontinuous and $K_{h_t}(\mu_t,\nu_t)<\infty$ for each~$t$.
Then for the measures $\mu_t^n$ and $\nu_t^n$ from the previous lemma applied to
$\mu_t$ and $\nu_t$ we have
$$
K_{h_t}(\mu_t,\nu_t)=\lim\limits_{n\to\infty} K_{h_t}(\mu_t^n,\nu_t^n)
\quad \forall\, t\in T.
$$
\end{lemma}
\begin{proof}
Let $t$ be fixed. We have
$$\mu_t^n\le p_n(t)\mu_t^{n+1}
\quad \hbox{and}\quad
\mu_t^n\le q_n(t)\mu_t,
$$
where $q_n(t)>1$ and $p_n(t)>1$ are numbers converging to~$1$.
Hence there is a finite limit $\lim\limits_{n\to\infty} K_{h_t}(\mu_t^n,\nu_t^n)\le
K_{h_t}(\mu_t,\nu_t)$. We now prove the opposite inequality.
Let $\sigma_t^n\in\Pi(\mu_t^n,\nu_t^n)$
be optimal measures for~$h_t$. Both sequences $\{\mu_t^n\}$ and $\{\nu_t^n\}$ are uniformly tight,
hence  $\{\sigma_t^n\}$ is also uniformly tight and
contains a weakly convergent subsequence, which we denote by the same
indices. Let $\sigma_t$ be its limit. Clearly,
$\sigma_t\in \Pi(\mu_t,\nu_t)$. The integral of $h_t$ against $\sigma_t$
does not exceed the liminf of the integrals of $h_t$ against the measures~$\sigma_t^n$
(see \cite[Corollary~4.3.4]{B18}), which is
exactly the limit of $K_{h_t}(\mu_t^n,\nu_t^n)$.
\end{proof}

Let us recall the following classical result
going back to Novikoff and Kunugui,
see \cite[p.~224, 225]{Del} (or \cite[Theorem~18.18]{Kech}, where $X$ is a standard Borel space).

\vskip .1in
{\bf Theorem A.}
{\it Let $X$ be a Souslin space, $Y$ a Polish space, and $B\subset X\times Y$ a~Borel set
such that for all $x\in X$ the sections
$B_x$  are $\sigma$-compact {\rm(}countable unions of compact sets{\rm)}.
Then $B$ admits a Borel uniformization, which means that the projection $\pi_X(B)$
of $B$ on $X$ is a Borel set and there is a Borel mapping $$f\colon \pi_X(B)\to Y$$ whose graph
is contained in~$B$.}
\vskip .1in

There is also another classical result with somewhat different assumptions
(see, e.g., \cite[Theorem~6.9.3 and Corollary~6.9.4]{B07}).

\vskip .1in
{\bf Theorem B.}
{\it Let $(T,\mathcal{T})$ be a general measurable space, let $E$ be a Polish space, and let
$\Psi$ be a mapping on $T$ with values in the set of nonempty closed subsets of $E$ that is
measurable in the following sense: for every open set $U\subset E$, the projection
of the set $\{(t,x)\colon x\in \Psi(t)\cap U\}$ on $T$ belongs to~$\mathcal{T}$.
Then there is a $(\mathcal{T},\mathcal{B}(E))$-measurable mapping
$\zeta\colon T\to E$ with $\zeta(t)\in \Psi(t)$ for all~$t$, i.e., a
$(\mathcal{T},\mathcal{B}(E))$-measurable selection. Moreover, there is a sequence of
$(\mathcal{T},\mathcal{B}(E))$-measurable mappings
$\zeta_n\colon T\to X$ such that the sequence $\{\zeta_n(t)\}$ is dense in $\Psi(t)$ for each~$t$.}
\vskip .1in

The difference between the two theorems is that in the latter the space $T$ is more general,
but the hypotheses include the measurability of the aforementioned projections, while
in the former this measurability follows from other assumptions
(here we consider $\Psi(x)=B_x$ in order to compare the settings). Indeed, to see this we observe that
it suffices to verify the required measurability for closed sets~$U$ (since any open set in a Polish
space is some countable union of closed sets). But then the sections of $B\cap (X\times U)$ remain
$\sigma$-compact, so the projection   remains Borel.
Note that in Theorem~A there is also a sequence of Borel mappings
$f_n\colon \pi_X(B)\to Y$ such that $\{f_n(x)\}$ is dense in $B_x$ for each $x\in \pi_X(B)$.

Thus, Theorem~B is formally more general (but to see this we need Theorem~A),
however, practically the most difficult part
is to verify the measurability of projections (and the proof of Theorem~A is more difficult).
So our main tool will be Theorem~A. It should be noted that Theorem~A is not valid for arbitrary
measurable spaces in place of Souslin spaces (it  fails even for co-analytic sets in $[0,1]$ and
single-valued sections).

Finally, let us mention yet another known result (see \cite[Theorem~6.9.5]{B07})
in which the assumptions are weaker,
but also the guaranteed measurability of selections is weaker.

\vskip .1in
{\bf Theorem C.}
{\it Suppose that $T$ and $E$ are Souslin spaces.
Let $\Psi$ be a   multivalued mapping  from
$T$ to the set of nonempty subsets of $E$ such that its graph
$$
\Gamma_\Psi=\{(t,u)\colon t\in T\, u\in \Psi(t)\}
$$
is a Souslin set in~$T\times E$.
Then, there exists a sequence of selections~$\zeta_n$ that are
measurable as mappings from  $(T,\sigma(\mathcal{S}(T)))$ to
$(E,\mathcal{B}(E))$ and, for every~$t\in T$, the sequence
$\{\zeta_n(t)\}$ is dense in the set~$\Psi(t)$.}

\vskip .1in

In our situation, a typical application of these results is this.

The set-valued mapping
$$
(\mu,\nu)\mapsto \Pi(\mu,\nu)
$$
from $\mathcal{P}(X)\times \mathcal{P}(Y)$ to the set of nonempty compact
subsets of $\mathcal{P}(X\times Y)$ is measurable in the aforementioned sense.
Alternatively, we can apply Theorem~A by using the easy fact that the set $B$ of triples
$(\mu,\nu,\sigma$) in  $\mathcal{P}(X)\times\mathcal{P}(Y)\times \mathcal{P}(X\times Y)$
such that $\sigma\circ\pi_X^{-1}=\mu$ and $\sigma\circ\pi_Y^{-1}=\nu$
is Borel and its sections
$B_{\mu,\nu}$ are compact. Hence there is a sequence
of Borel mappings
$$
\Phi_n\colon \mathcal{P}(X)\times \mathcal{P}(Y)\to \mathcal{P}(X\times Y)
$$
such that the sequence $\{\Phi_n(\mu,\nu)\}$ is dense in $\Pi(\mu,\nu)$
for all $\mu$ and $\nu$.

Let $t\mapsto \mu_t$ and  $t\mapsto \nu_t$ be measurable mappings from $(T,\mathcal{T})$
to the spaces $\mathcal{P}(X)$ and $\mathcal{P}(Y)$
of Borel probability measures on Polish spaces~$X$ and~$Y$.
Then  there is a sequence of measurable mappings $\Psi_n\colon T\to \mathcal{P}(X\times Y)$
such that $\Psi_n(t)\in \Pi(\mu_t,\nu_t)$ and the sequence $\{\Psi_n(t)\}$
is dense in $\Pi(\mu_t,\nu_t)$ for each~$t$. To this end, we set
$\Psi_n(t):=\Phi_n(\mu_t,\nu_t)$.

Suppose now that
$(x,y)\mapsto h(t,x,y)$ is continuous for each fixed~$t\in T$ (the case of Theorem~\ref{tmain1}).
Then the function
$$
K(t)=K_{h_t}(\mu_t,\nu_t)
$$
is measurable on~$T$, which proves the first assertion of Theorem~\ref{tmain1}. Indeed,
$$
K(t)=\inf_n \int_{X\times Y} h(t,x,y)\, \Psi_n(t)(dx dy).
$$

Let now
$$
M_t:=\biggl\{\sigma\in \Pi(\mu_t,\nu_t)\colon
\int h(t,x,y)\, \sigma(dx dy)=K(t)\biggr\}.
$$
Each set $M_t$ is compact in $\Pi(\mu_t,\nu_t)\subset \mathcal{P}(X\times Y)$.
Once we have the measurability of the set-valued mapping $t\mapsto M_t$ we can use selection theorems.
However, the problem is to verify this measurability. This will be done below
for Souslin spaces~$T$ in order to have the Borel measurability.
However, if we are satisfied with the measurability with respect to
the  $\sigma$-algebra $\sigma(\mathcal{S}(\mathcal{T}))$ on~$T$, then
we can apply Theorem~B to this larger $\sigma$-algebra. The hypothesis of Theorem~B is fulfilled.
Indeed, let $U$ be an open set in~$\mathcal{P}(X\times Y)$. The set of pairs
$(t,\sigma)$ in $T\times \mathcal{P}(X\times Y)$, where
$\sigma\in \Pi(\mu_t,\nu_t)$ and the integral of $h_t$ against $\sigma$ is $\mathcal{T}$-measurable,
is contained in $\mathcal{T}\otimes \mathcal{B}(\mathcal{P}(X\times Y))$ by the $\mathcal{T}$-measurability
of~$K$. Hence the intersection of this set with $T\times U$
is also in~$\mathcal{T}\otimes \mathcal{B}(\mathcal{P}(X\times Y))$.
Therefore, the projection of this intersection belongs to $\mathcal{S}(\mathcal{T})$
by a known result (see \cite[Corollary~6.10.10]{B07}).

Finally, the proof of Theorem~\ref{tmain3} is completely analogous,
the only difference is that now we apply Theorem~C: the set of pairs $(t,\sigma)$
in $T\times \mathcal{P}(X\times Y)$ such that $\sigma\in \Pi_t(\mu_t,\nu_t)$ is
Borel as  above. Hence there is a sequence of $\mathcal{S}(T)$-measurable
mappings $\Psi_n\colon T\to \mathcal{P}(X\times Y)$ such that the sequence $\{\Psi_n(t)\}$
is dense in $\Pi_t(\mu_t,\nu_t)$, so $K(t)$ equals the infimum of the sequence of
integrals of $h_t$ against~$\Psi_n(t)$. Once we know that $K(t)$ is Borel measurable,
the same reasoning applies to the set of pairs  $(t,\sigma)$
with the additional restriction that the integral of $h_t$ against $\sigma$ equals~$K(t)$,
but this restriction determines a Borel set.

\begin{lemma}\label{lem3}
Suppose that $Z$ is a Borel set in a complete separable metric space
with a metric~$d$, $T$ is a Souslin space, and
$h\colon T\times Z\to [0,2]$ is a Borel function that is lower semicontinuous
in the second variable and has the following property{\rm:}
 for every $t$ there is a compact set $Z_t\subset Z$
such that $h(t,z)\in [0,1)$ for all $z\in Z_t$ and
$h(t,z)=2$ for all $z\in Z\backslash Z_t$.
Then there is a sequence of Borel mappings
$\psi_j\colon T\to Z$ such that
\begin{equation}\label{appr1}
\inf\{h(t,z)+d(x,z)\colon z\in Z\}=\inf_j [h(t,\psi_j(t))+d(x,\psi_j(t))]
\quad \forall\, x\in Z, t\in T.
\end{equation}
\end{lemma}
\begin{proof}
We  consider the sets
$$
S_{k,m}=\{(t,z)\in T\times B_m\colon h(t,z)\in U_k\},
$$
where $\{U_k\}$ is the sequence of all rational semiclosed intervals $(a,b]$ in $[-1,1]$
and $\{B_m\}$ is the sequence of all closed balls with positive rational radii centered at the
points of a fixed countable dense set $\{z_l\}$ in~$Z$.
The sets $S_{k,m}$ are Borel. We take into account only nonempty sets~$S_{k,m}$.
Note that if $(t,z)\in S_{k,m}$, then $z$ must belong to~$Z_t$, since $U_k\subset [-1,1]$ and
$h(t,\cdot)=2$ outside~$Z_t$.
For each~$t\in T$, the section
$$
S_{k,m}^t=\{x\colon (t,x)\in S_{k,m}\}
$$
 is the difference of two compact sets by the lower
semicontinuity of $h$ in the second argument and the inclusion $S_{k,m}^t\subset Z_t$.
Hence this section is $\sigma$-compact. Therefore,
by Theorem~A stated above,
the projection of $S_{k,m}$ onto~$T$, denoted by~$T_{k,m}$, is a Borel set and there
is a Borel mapping $\psi_{k,m}\colon T_{k,m}\to Z$
such that $\psi_{k,m}(t)\in S_{k,m}^t$ for each~$t\in T_{k,m}$.
Outside $T_{k,m}$ we set $\psi_{k,m}(t)=z_1$.
Let us add to this sequence the countable family of constant mappings with values in $\{z_l\}$.
Finally, we renumber the obtained collection
by using a single index~$j$.

We now verify (\ref{appr1}). Since both sides of (\ref{appr1})
are Lipschitz in~$x$, it suffices to show that they
coincide for all $x\in \{z_l\}$.
Fix $t\in T$, $x=z_l$ and $\varepsilon>0$.
Take $z\in Z$ for which $h(t,z)+d(x,z)-\varepsilon/2$
is less than the left-hand side of  (\ref{appr1}).
If $x\not\in Z_t$, then either the left-hand side equals~$2$ and the minimum is attained at
$z=z_l$, so the corresponding constant function works,
or $z\in Z_t$, because $h(t,z)=2$ outside~$Z_t$. If $x\in Z_t$, then we also have $z\in Z_t$.
We show that there are numbers $k$ and $m$ such that
the left-hand side of  (\ref{appr1}) is larger than
$$
h(t,\psi_{k,m}(t))+d(x,\psi_{k,m}(t))-\varepsilon.
$$
To this end,
we find $k$ and $m$ for which $h(t,z)\in U_k$ and $z\in B_m$, moreover, we pick $k$ and $m$ such that
the length of $U_k$ and the diameter of $B_m$ are less than $\varepsilon/8$.
Then for the corresponding $\psi_{k,m}(t)$ we have
$\psi_{k,m}(t)\in B_m$, $h(t,\psi_{k,m}(t))\in U_k$, so that
$$
h(t,\psi_{k,m}(t))+d(x, \psi_{k,m}(t))< h(t,z)+d(x,z)+\varepsilon/4<
\inf\{h(t,z)+d(x,z)\colon z\in Z\}+\varepsilon,
$$
which completes the proof.
\end{proof}

\begin{lemma}\label{lem-appr}
Under the hypotheses of the previous lemma, there is a sequence of Borel
functions $h_n\colon T\times Z\to [0,2]$ such that
$h_n\le h_{n+1}$, $h(t,z)=\lim\limits_{n\to \infty} h_n(t,z)$, and
the functions $z\mapsto h_n(t,z)$ are bounded Lipschitz for each $t$.
\end{lemma}
\begin{proof}
There is a classical construction for approximations:
$$
h_n(t,z)=\inf\{ h(t,y)+n d(z,y), \ y\in Z\}.
$$
The function $h_n$ is Lipschitz in $z$ and $h_n\le h$. Its Borel measurability in $t$ follows by the
previous lemma applied to the metric~$nd$, so $h_n$ is jointly Borel measurable.
\end{proof}

\begin{remark}\label{rem-appr}
\rm
The assumption that $T$ is a Souslin space has been used
in the previous two lemmas to cover the case of lower semicontinuous
functions. If the functions $h_t$ are continuous for each~$t$ and $h$ is
measurable on~$T\times Z$ (not necessarily bounded), then both lemmas
are valid for arbitrary measurable spaces~$(T,\mathcal{T})$, since the approximations
$$
h_n(t,z)=\inf_k [h(t,y_k)+n d(z,y_k)],
$$
where $\{y_k\}$ is a fixed sequence dense in~$Z$, coincide with the functions
defined above by the infimum over the whole space and are Lipschitz. Replacing them
by $\min(h_n,n)$ we obtain bounded Lipschitz functions increasing to~$h$ and measurable
on~$T\times Z$.
\end{remark}

\begin{lemma}\label{lem-lim}
Suppose that lower semicontinuous cost functions $h_n\ge 0$
increase pointwise to a function $h$
for which $K_h(\mu,\nu)<\infty$. Let $\pi_n\in \Pi(\mu,\nu)$ be optimal measures  for $h_n$
converging weakly to a Radon measure $\pi$. Then
$\pi$ is an optimal measure for the triple $h,\mu,\nu$.
In addition, $K_h(\mu,\nu)=\lim\limits_{n\to\infty}K_{h_n}(\mu,\nu)=\lim\limits_{n\to\infty} I_{h_n}(\pi_n)$.
\end{lemma}
\begin{proof}
 For continuous cost functions this assertion is simple.
 For the reader's convenience, we include the proof. Clearly, $\pi\in \Pi(\mu,\nu)$.
The sequence $\{\pi_n\}$ is uniformly tight,
so, given $\varepsilon>0$, there is a compact set $K$ with $\pi(K)>1-\varepsilon$,
$\pi_n(K)>1-\varepsilon$ for all $n$.
Enlarging $K$ we can assume that the integral of $h$ over the complement of $K$ with respect
to $\pi$ is less than $\varepsilon$.
On $K$ convergence is uniform by Dini's theorem. Then
$$
|I_h(\pi)-K_{h_n}(\mu,\nu)|\le 2\varepsilon
$$
for large $n$. Hence the numbers $K_{h_n}(\mu,\nu)$
increase  to $I_h(\pi)$.
Since
$$
K_{h_n}(\mu,\nu)\le K_{h}(\mu,\nu)\le I_h(\pi),
$$
 we have
$I_h(\pi)=K_{h}(\mu,\nu)$. This reasoning also
applies to the case where only the function $h$ is continuous, but all $h_n$ are lower semicontinuous
(to apply Dini's theorem, we need the upper semicontinuity of the functions $h-h_n$).

Our next step is to observe that  for lower semicontinuous $h$
the quantity $K_h(\mu,\nu)$ coincides with the supremum of
$K_w(\mu,\nu)$ over bounded continuous cost functions $w\ge 0$ such that
$w(x,y)\le h(x,y)$ for all $x$ and $y$. This follows by
the Kantorovich duality: for each $\varepsilon>0$ there are functions
$\varphi\in C_b(X)$ and $\psi\in C_b(Y)$ such that
$$
\varphi(x)+\psi(y)\le h(x,y)
$$
for all $x$ and $y$ and
$$
\int\varphi\, d\mu+\int\psi \, d\nu \ge
K_h(\mu,\nu)-\varepsilon.
$$
We now take $w(x,y)=\max(\varphi(x)+\psi(y),0)$. Since
$w(x,y)\ge \varphi(x)+\psi(y)$,
the integral of $h$ against any measure in $\Pi(\mu,\nu)$ is at least $K_h(\mu,\nu)-\varepsilon$.
Hence we have $K_w(\mu,\nu)\ge K_h(\mu,\nu)-\varepsilon$.

It follows that there is a pointwise increasing sequence of nonnegative functions $w_n\in C_b(X\times Y)$
 such that $w_n(x,y)\le h(x,y)$ and $K_{w_n}(\mu,\nu)\to K_h(\mu,\nu)$.
 Such functions can be found converging to $h$, since
 there is a sequence of bounded continuous functions $u_n\ge 0$ increasing to $h$,
 so we can take $\max(w_n,u_n)$ and observe that
 $K_{w_n}(\mu,\nu)\le K_{\max(w_n,u_n)}(\mu,\nu)  \le K_h(\mu,\nu)$.

Let us show that there is no gap between $K_h(\mu,\nu)$ and the limit of $K_{h_n}(\mu,\nu)$
in the general case.
Let $\varepsilon>0$. Take a function $w\in C_b(X\times Y)$ with $0\le w\le h$ and
$K_w(\mu,\nu)\ge K_h(\mu,\nu)-\varepsilon$.

The sequence of bounded lower semicontinuous
functions $v_n=\min(w, h_{n})$ increases pointwise to the bounded continuous function $w$.
Hence by the previous step
$$
K_{v_n}(\mu,\nu)\to K_{w}(\mu,\nu)\ge K_h(\mu,\nu)-\varepsilon.
$$
Since $K_{h_{n}}(\mu,\nu)\ge K_{v_n}(\mu,\nu)$, we conclude that
$K_{h_n}(\mu,\nu)\to K_h(\mu,\nu)$.

It remains to show that $K_h(\mu,\nu)$ coincides with $I_h(\pi)$. Otherwise
 for some $\delta>0$ we have
$I_h(\pi)>K_h(\mu,\nu)+\delta$. Using the functions $w_n$ constructed above,
we obtain a number $N$ such that
$$
\int w_N\, d\pi> K_h(\mu,\nu)+\delta/2.
$$
Hence
$$
\int w_N\, d\pi_n> K_h(\mu,\nu)+\delta/2
$$
for all $n$ large enough. Since $w_N$ is bounded and $\{\pi_n\}$ is uniformly tight,
there is a compact set $K$ such that
$$
\int_K w_N\, d\pi_n> K_h(\mu,\nu)+\delta/4
$$
for all $n$ large enough.
The functions $\min(h_n,w_N)$ are lower semicontinuous and increase to
the continuous function~$w_N$. Hence convergence
is uniform on~$K$. Therefore,
$$
\int_K \min(h_n,w_N)\, d\pi_n> K_h(\mu,\nu)+\delta/8
$$
for all $n$ large enough.  This yields the bound
$$
K_{h_n}(\mu,\nu)=\int_{X\times Y} h_n\, d\pi_n
\ge
\int_K h_n\, d\pi_n
\ge
\int_K \min(h_n,w_N)\, d\pi_n
> K_h(\mu,\nu)+\delta/8,
$$
which is a contradiction.
\end{proof}

\begin{lemma}\label{lastlem}
Suppose that in the situation of Theorem~{\rm\ref{tmain2}}
the measurability of $t\mapsto K_{h_t}(\mu_t,\nu_t)$ is given in advance.
Then the assertion about the existence of a Borel version of $\sigma_t$ is true.
\end{lemma}
\begin{proof}
Now by assumption the function
$$
K(t)=K_{h_t}(\mu_t,\nu_t)
$$
is measurable on~$T$. Let
$$
M_t:=\biggl\{\sigma\in \Pi(\mu_t,\nu_t)\colon
\int h(t,x,y)\, \sigma(dx dy)=K(t)\biggr\}.
$$
Each set $M_t$ is compact in $\Pi(\mu_t,\nu_t)\subset \mathcal{P}(X\times Y)$,
because if measures $\sigma_n\in M_t$ converge weakly to a measure~$\sigma$,
then $\sigma\in \Pi(\mu_t,\nu_t)$ and the integral of $h_t$ against $\sigma$ cannot be larger
than~$K(t)$ by the lower semicontinuity of~$h_t$, but obviously it cannot be smaller than $K(t)$
by the definition of~$K(t)$.

By the Borel measurability of the function
$t\mapsto K(t)$ and the Borel measurability of the function
$$
(t,\sigma)\mapsto \int_{X\times Y} h(t,x,y)\, \sigma(dx dy)
$$
on $T\times \mathcal{P}(X\times Y)$, which follows by the joint measurability of~$h$
(see \cite[Theorem~5.8.4]{B18}),
the set
$$
B=\biggl\{(t,\sigma)\colon \sigma\in \mathcal{P}(X\times Y), \ \sigma\in \Pi(\mu_t,\nu_t),
\
\int h(t,x,y)\, \sigma(dx dy)=K(t)\biggr\}
$$
is Borel in $T\times \mathcal{P}(X\times Y)$ and $M_t$ is its section at~$t$.
Hence again Theorem~A applies.
\end{proof}

\begin{lemma}\label{select1}
Let $(T,\mathcal{T})$ be a measurable space, let $E$ be a
completely regular Souslin space, and
let $u_{n}\colon T\to E$ be a sequence of $\mathcal{T}$-measurable
mappings such that the sequence  $\{u_{n}(t)\}$
has compact closure
for every fixed~$t\in T$. Then there is a sequence of $\mathcal{T}$-measurable
functions $t\mapsto \eta_k(t)$ with values in $\mathbb{N}$ such that, for every~$t$,
the numbers  $\eta_k(t)$ increase to infinity and the sequence $\{u_{\eta_k(t)}(t)\}$ converges
to some point $u(t)$ such that the mapping $t\mapsto u(t)$ is $\mathcal{T}$-measurable.
\end{lemma}
\begin{proof}
There is a continuous injection of $E$ into $[0,1]^\infty$,
so we can consider $E$ as a set in $[0,1]^\infty$ with a stronger Souslin topology.
Points of $[0,1]^\infty$ will be written as $x=(x^1,x^2,\ldots)$.
It suffices to pick increasing numbers $\eta_k(t)$ measurably
in~$t$ in such a way that, for each~$j$ and~$t$,
 the sequence of numbers $u_{\eta_k(t)}^j(t)$ will converge.
 Indeed, this convergence implies that the sequence $\{u_{k(t)}(t)\}$ cannot have
 different limit points, but by the compactness of the closure this sequence must
 have limit points, so it follows that the whole
 sequence converges.

 We construct $\eta_k(t)$ inductively. By the measurability of $u_{n}$
 the functions
 $$
 L_j(t)=\limsup_{n\to\infty} u_{n}^j(t)
 $$
 are $\mathcal{T}$-measurable.
  Let $\eta_1^1(t)$ be the minimal number $n$ such that
 $$
|u_{n}^1(t)- L_1(t)|<1.
 $$
 This number measurably depends on $t$, because
  \begin{align*}
 &\biggl\{t\in T\colon \eta_1^1(t)=m\biggr\}
 \\
 &=\biggl\{t\colon
 |u_{n}^1(t)- L_1(t)|\ge 1, n=1,\ldots, m-1,
 |u_{m}^1(t)- L_1(t)|<1\biggr\}.
 \end{align*}
 Assuming that $\eta_k^1(t)$ is already defined and $\mathcal{T}$-measurable, we take for
 $\eta_{k+1}^1(t)$ the minimal number $n$ such that $n>\eta_k^1(t)$ and
 $$
|u_{n}^1(t)- L_1(t)|<\frac{1}{k+1}.
 $$
 As above, the function $\eta_{k+1}^1$ is $\mathcal{T}$-measurable. It follows that
 the first coordinates of $u_{\eta_k^1(t)}$ converge to~$L_1(t)$.

 The second step is to pick a subsequence in $\{\eta_k^1(t)\}$ for which the second coordinates
 will converge to~$L_2(t)$. To this end, we take for $\eta_1^2(t)$
 the minimal number $n>\eta_1^1(t)$ among the numbers $\eta_{k}^1(t)$ such that
 $$
|u_{n}^2(t)- L_2(t)|<1.
 $$
 We have
 \begin{align*}
 &
 \{t\in T\colon \eta_1^2(t)=m\}
 \\
 &=\{t\colon
 |u_{\eta_1^n(t)}^2(t)- L_2(t)|\ge 1, n=1,\ldots, m-1,
 |u_{\eta^1_m(t)}^2(t)- L_2(t)|<1\},
 \end{align*}
 which shows that $\eta_1^2$ is $\mathcal{T}$-measurable. We proceed inductively and find
 $\mathcal{T}$-measurable functions $\eta_{k}^2$ such that $\eta_k^2(t)$ is the minimal
 number in $\{\eta_k^1(t)\}$ for which the difference between $L_2(t)$ and the second
coordinate of $u_{\eta_n^1(t)}(t)$ becomes less than~$1/k$.

 We continue this process inductively and obtain embedded subsequence $\{\eta_k^m(t)\}$
 such that the functions $\eta_k^m$ are $\mathcal{T}$-measurable and the $m$th coordinates of
  $u_{\eta_k^m(t)}(t)$ converge to~$L_m(t)$. For the diagonal sequence
 $\eta_k^k(t)$ we have convergence of all coordinates, which proves convergence
 of~$u_{\eta_k^k(t)}(t)$.
\end{proof}

\begin{corollary}\label{c-select1}
Let $(T,\mathcal{T})$ be a measurable space, let $X$ be a
completely regular Souslin space, and
let $t\mapsto \mu_{t,n}$, $T\to \mathcal{M}(X)$ be a sequence of $\mathcal{T}$-measurable
mappings such that the sequence of measures $\{\mu_{t,n}\}$
has weakly compact closure {\rm(}for example, is uniformly tight{\rm)}
for every fixed~$t\in T$. Then there is a sequence of $\mathcal{T}$-measurable
functions $t\mapsto \eta_k(t)$ with values in $\mathbb{N}$ such that, for every~$t$,
the numbers  $\eta_k(t)$ increase to infinity and the sequence of measures $\mu_{t,\eta_k(t)}$ converges
to some measure $\mu_t$ such that $t\mapsto \mu_t$ is $\mathcal{T}$-measurable.
\end{corollary}
\begin{proof}
The previous lemma applies, since the space of measures on $X$ with the weak topology is also
Souslin.
\end{proof}

\begin{proof}[Proof of Theorem~\ref{tmain1}]
By Corollary~\ref{c-select1} for completing the proof of Theorem \ref{tmain1}
it suffices to find approximate $\mathcal{T}$-measurable
solutions $\sigma_{t,n}$ with
$I_h(\sigma_{t,n})\to K(t)$ for each~$t$. To this end, we find
$\mathcal{T}$-measurable
solutions $\pi_{t,n}$ for bounded Lipschitz cost functions $h_n$ increasing to $h$
and constructed according to Remark~\ref{rem-appr}.
Therefore, the general case reduces to the case in which
every function $h_t$ is bounded by $1$ and Lipschitz with constant~$1$.
Moreover, by Lemma~\ref{lem1} and Lemma~\ref{lem2} it suffices to consider the case
in which the measures $\mu_t$ and $\nu_t$ have compact supports, so that for each $t$
there is a compact set $S_t$ on which all measures from $\Pi(\mu_t,\nu_t)$ are
concentrated and $S_t$ depends on $t$ measurably.

Let us consider the space $\mathcal{K}(X\times Y)$ of nonempty compact subsets
of $X\times Y$ with the Hausdorff distance $d_H$ introduced in~Section~2.
This space is separable,
hence there is a sequence of compacts sets $Q_j$ dense in the union of~$S_t$.
Let fix $n$  and consider the sets
$$
T_j=\{t\in T\colon {\rm dist}_H(S_t,Q_j)\le 1/n\}.
$$
Note that $T_j\in \mathcal{T}$ (this follows from the proof of Lemma~\ref{lem1}).
The set of $1$-Lipschitz functions on $Q_j$ with values in $[0,1]$ is compact
in the sup-norm, hence there is a sequence $h_{j,m}$ dense in it. Each function
$h_{j,m}$ has an extension (denoted by the same symbol) to all of~$X\times Y$ with
values in $[0,1]$ and $1$-Lipschitz.

We further define the sets
$$
T_{j,m}=\{t\in D_j\colon \sup_{(x,y)\in Q_j} |h_t(x,y)-h_{j,m}(x,y)|\le 1/n\}.
$$
The supremum can be taken over a countable set dense in $Q_j$, hence $T_{j,m}\in \mathcal{T}$.
Using these sets we obtain a partition of $T$
into nonempty disjoint sets $D_k\in \mathcal{T}$
with the following property: for each $D_k$ there are numbers $j$ and $m$
such that ${\rm dist}_H(S_t,Q_j)\le 1/n$ and $\sup_{(x,y)\in Q_j}
|h_t(x,y)-h_{j,m}(x,y)|\le 1/n$ for all $t\in D_k$.
In every set $D_k$ take a point $t_k$.
The cost function $h_{t_k}$ differs from any other cost function $h_t$ with $t\in D_k$
by at most~$3/n$ on the set~$S_t$. Indeed, if $(x,y)\in S_t$, then
we can find $(u,v)\in Q_j$ with $d((x,y),(u,v))\le 1/n$.
Since on $Q_j$ the functions $h_t$ and $h_{t_k}$ differ by at most~$1/n$,
we have
\begin{multline*}
|h_t(x,y)-h_{t_k}(x,y)|
\\
\le
|h_t(x,y)-h_t(u,v)|+|h_t(u,v)-h_{t_k}(u,v)|+|h_{t_k}(u,v)-h_{t_k}(x,y)|
\le 3n^{-1}.
\end{multline*}
Finally, on each $D_k$ we solve the Kantorovich problem with the cost function $h_{t_k}$
independent of $t$ and the original marginals. Hence there is a solution $\pi_t^k\in \Pi(\mu_t,\nu_t)$
that is $\mathcal{T}$-measurable. Clearly,
$|K_{h_t}(\mu_t,\nu_t)-K_{h_k}(\mu_t,\nu_t)|\le 3/n$ for all $t\in D_k$.
Therefore, on all of $T$ we obtain the desired approximation.
\end{proof}

\begin{proof}[Proof of Theorem~\ref{tmain2}]
By Lemma~\ref{lastlem} it suffices to prove the Borel measurability of the transportation
cost~$K_t=K_{h_t}(\mu_t,\nu_t)$.
Using Lemma~\ref{lem-lim} and the truncations
$\min(h_t,N)$ we can pass to uniformly bounded cost functions. So we can assume that $h_t<1$.
Lemma~\ref{lem2} reduces the assertion to the case of measures $\mu_t^n$ and
$\nu_t^n$ with compact supports $Z_1^n(t)$ and $Z_2^n(t)$.
The value of the cost does not change if we redefine $h_t$ outside
$Z_1^n(t)\times Z_2^n(t)$ by the value~$2$. Since the set
 $\bigcup_t (\{t\}\times Z_1^n(t)\times Z_2^n(t))$
belongs to $\mathcal{T}\otimes \mathcal{B}(X)\otimes \mathcal{B}(Y)$ by Lemma~\ref{lem1},
this new  cost function is Borel. It is readily seen that it is lower semicontinuous.

Now we are in the situation of Lemma~\ref{lem-appr}. Therefore,
Lemma~\ref{lem-lim} further reduces everything to continuous
cost functions. This case is covered by the first (and easy)
part of Theorem~\ref{tmain1}.
\end{proof}

\begin{remark}\label{rem-Zh}
\rm
(i)
As already noted in the introduction,
Zhang \cite{Zhang} proved that if cost functions $h_t$ are continuous, then
the space $M$ of nonnegative continuous cost functions can be regarded as a
parametric space and equipped with its natural Borel $\sigma$-algebra
(generated by the metric introduced above) and
the set-valued mapping $(h,\mu,\nu)\mapsto {\rm Opt}(h,\mu,\nu)$
has a Borel measurable selection.

However, it is not clear how this can be applied to the assertion
announced in \cite{Zhang} that a measurable selection exists for any
parametric measurable space~$(T,\mathcal{T})$. The point is that the mapping $t\mapsto h_t$
with values in $M$ generated by a function $h$ Borel measurable
in $t$ can fail to be measurable when $M$ is equipped with the Borel $\sigma$-algebra.
For example, this happens if $T=C_b(B\times B)$ with its sup-norm,
where $B$ is the unit ball in~$l^2$, $\mathcal{T}$ is generated by evaluation
functionals $t\mapsto t(x,y)$,
$X=Y=B$, and  $h(t,x,y)=t(x,y)$. Here $h$ is  bounded and continuous in $(x,y)$
and $\mathcal{T}$-measurable, but $t\mapsto h_t$ is not measurable
with values in $C_b(B\times B)$ equipped with the Borel $\sigma$-algebra.

To see this, let us observe that the Borel
$\sigma$-algebra of the space $C_b(\mathbb{N})$
is not countably generated, because its cardinality is greater
than that of the continuum.
Indeed, this space contains a closed discrete set of cardinality of the continuum; all subsets of this set
are also closed. It follows that $\mathcal{B}(C_b(\mathbb{N}))$
is not generated by the evaluation functionals $f\mapsto f(n)$.
The same is true for any metric space containing a discrete countable subset, hence for any
noncompact metric space. Similarly, the Borel $\sigma$-algebra
of the metric space $M=C(X\times Y)$ mentioned in the introduction is not generated
by evaluation functionals if the balls in $X\times Y$ are not compact.

However, for any compact metric space $K$ the Borel $\sigma$-algebra of the space $C_b(K)$ is
generated by the evaluation functionals $f\mapsto f(k)$,
because this space is separable and these functionals separate its points.
Hence the proof in \cite{Zhang} for general $(T,\mathcal{T})$ is correct
if $X$ is a locally compact Polish space. We have not succeeded to fix the general
case in a simple way and needed several steps.
Recall also that for lower semicontinuous cost functions
we still assume that $T$ is Souslin.

(ii)
The question also arises whether
Theorem~\ref{tmain2} extends to Souslin spaces $X$ and~$Y$. A~major problem is to extend
Lemma~\ref{lem2} to Souslin spaces~$X$. Suppose that for a lower semicontinuous $h_t$ and
Souslin spaces $X$ and $Y$ we know that there are measurable set-valued mappings
$t\mapsto Z_n(t)$ as in Lemma~\ref{lem2}. We take a bounded continuous metric $d$
on $X\times Y$ and observe that the functions
$h_k(t,x,y)=\inf \{h(t,u,v)+kd((x,y), (u,v))\colon (u,v)\in Z_n(t)\}$
increase on $Z_n(t)$ to $h(t,x,y)$, because on $Z_n(t)$ the topology of $X\times Y$
is metrizable by~$d$ by compactness. Moreover, the assumed measurability of $Z_n(t)$
ensures (for each fixed~$n$)
the existence of a sequence of measurable mappings $\xi_j\colon T\to X\times Y$
such that $\xi_j(t)\in Z_n(t)$ and $Z_n(t)$ is the closure of $\{\xi_j(t)\}$.
So the infimum defining $h_k(t,x,y)$ can be evaluated over $\{\xi_j(t)\}$, which shows
the measurability of~$h_k(t,x,y)$.

Of course, if we agree to leave the safe area of Borel measurability,
for Souslin spaces it is possible
to impose the following stronger condition on $\mu_t$ and~$\nu_t$: let these mappings
be measurable when $X$ and $Y$ are equipped with the $\sigma$-algebras
$\sigma(\mathcal{S}(X))$ and
$\sigma(\mathcal{S}(X))$. Then $K(t)$ is $\sigma(\mathcal{S}(T))$-measurable
and $\sigma_t$ can be made $\sigma(\mathcal{S}(T))$-measurable. Indeed, there are
continuous surjections $g_1\colon \mathbb{R}^\infty\to X$,
$g_2\colon \mathbb{R}^\infty\to Y$. Hence we obtain two families
$\mu_t^1=\mu_t\circ g_1^{-1}$, $\nu_t^2=\nu_t\circ g_2^{-1}$ of measures on~$\mathbb{R}^\infty$
that are $\sigma(\mathcal{S}(T))$-measurable. The obtained results apply to these measures
and the cost function $h^0(t,u,v)=h(t, g_1(u), g_2(u))$, which satisfies our hypotheses.
The corresponding transportation cost and optimal measures will be
$\sigma(\mathcal{S}(T))$-measurable. Then we take the images of optimal measures under
the mapping $(g_1,g_2)$.
\end{remark}

Closing this section we mention that similar results can be obtained
for the Kantorovich problem with density constraints studied by
Korman and McCann~\cite{KM} (see also~\cite{Doled}).
The density constraint is an additional requirement on admissible
optimal measures: in place of the set $\Pi(\mu,\nu)$ we consider its subset $\Pi^\theta(\mu,\nu)$
consisting of measures having densities with respect to a given measure $\lambda$
on $X\times Y$ bounded by a given nonnegative Borel function $\theta\in L^1(\lambda)$.
If  $\Pi^\theta(\mu,\nu)$ is not empty and the cost function is bounded and lower semicontinuous,
then the set of minimizing measures is not empty. A~straightforward modification
of the reasoning above shows that also in this case there is a measurable choice of
optimal measures depending on the parameter on which marginal measures and the cost function
depend measurably. In a separate paper we shall consider a more general situation where
the constraint $\theta$ and the reference measure $\lambda$ also depend on a parameter.

\section{The Skorohod parametrization with a parameter}

In this short section we consider another parametric problem in the same circle of ideas.
It was shown by Skorohod \cite{Sk} that for any weakly convergent sequence of
Borel probability measures $\mu_n$ on a complete separable metric space~$X$ there is a sequence
of Borel mappings $\xi_n\colon [0,1]\to X$ with $\mu_n=\lambda\circ \xi_n^{-1}$,
where $\lambda$ is Lebesgue measure,  converging almost everywhere.
This important result was generalized by Blackwell and Dubins~\cite{Black2}
and Fernique~\cite{Fernique88}, who proved that for every measure $\mu\in \mathcal{P}(X)$ there is a
Borel mapping $\xi_\mu\colon [0,1]\to X$ such that $\mu$ is the image
of Lebesgue measure $\lambda$ under~$\xi_\mu$ and measures $\mu_n$ converge weakly
to $\mu$ if and only if the mappings $\xi_{\mu_n}$ converge to $\xi_\mu$ almost everywhere.
A~topological proof of this result along with some generalizations was given in~\cite{BoKol}
(see also \cite{BBK01}, \cite{B07}, and \cite{B18} on this topic).
The purpose of this section is to verify that this topological proof actually
yields the following result.

\begin{theorem}
Let $X$ be a complete separable metric space.
For every measure $\mu\in \mathcal{P}(X)$
there is a Borel mapping $\xi_\mu\colon [0,1]\to X$
with $\mu=\lambda\circ \xi_\mu^{-1}$
such that the mapping
$(\mu,t)\mapsto \xi_\mu(t)$
is Borel measurable on $\mathcal{P}(X)\times [0,1]$ and measures
$\mu_n$ converge weakly to $\mu$ if and only if the mappings $\xi_{\mu_n}$
converge to $\xi_\mu$ almost everywhere.

Therefore, for any family
of measures $\mu_\omega\in\mathcal{P}(X)$ measurably depending on a parameter~$\omega$
from a measurable space~$(\Omega,\mathcal{A})$, the mapping
$(\omega,t)=\xi_{\mu_\omega}(t)$ with values in~$X$ is $\mathcal{A}\otimes\mathcal{B}[0,1]$-measurable.
\end{theorem}
\begin{proof}
We verify that the proof suggested in \cite{BoKol} and also presented
in \cite[\S8.5]{B07}
and \cite[\S2.6]{B18} gives the desired version. This proof is very simple. First we explicitly
define the desired mapping for the space~$X=[0,1]$:
$$
\xi_\mu(t)=\sup\{x\in [0,1]\colon \mu([0,x))\le t\}.
$$
It is shown in \cite[Theorem~2.6.4]{B18} that this is the desired parametrization.
We only need to show that $\xi_\mu(t)$ is jointly Borel measurable
on $\mathcal{P}([0,1])\times [0,1]$. Note that $\xi_\mu(t)$ is increasing and right-continuous
in~$t$.
It is known that if a function $\xi_\mu(t)$ is increasing and right-continuous in~$t$
for every fixed~$\mu$ and is Borel measurable in $\mu$ for each fixed~$t$, then it is jointly
Borel measurable. Indeed, it suffices to observe that it is the limit of the decreasing sequence
of functions $\xi_n(\mu,t)$ defined as follows: for each~$n$,
we partition $[0,1]$ into $2^n$ intervals $I_1=[0,2^{-n})$, $I_2=[2^{-n}, 2^{2-n}), \ldots,
I_{2^n}=[1-2^{-n}, 1]$ and set $\xi_n(\mu,t)=\xi_\mu(r_k)$ if $t\in I_k$ and $r_k$ is the right
end of~$I_k$.

The next step is to observe that once this theorem is established for some space~$X$,
it remains valid for every Borel subspace $E\subset X$. Indeed, every measure
$\mu\in\mathcal{P}(E)$ extends to a measure on $X$ by letting $\mu(X\backslash E)=0$.
We take a jointly Borel measurable mapping $(\mu,t)\mapsto \xi_\mu(t)$ for~$X$
and for measures concentrated on~$E$ redefine it by
$\eta_\mu(t)=\xi_\mu(t)$ if $\xi_\mu(t)\in E$
and $\eta_\mu(t)=x_0$ if $\xi_\mu(t)\not\in E$, where $x_0\in E$ is a fixed element.
Since $\xi_\mu(t)\in E$ for almost all~$t$ for $\mu$ concentrated on~$E$, we do not change
the image of Lebesgue measure. The obtained mapping is obviously Borel measurable
and gives the desired parametrization for~$\mathcal{P}(E)$.

It follows from the previous step that the theorem is true for the Cantor set~$C$.
It is known that every compact metric space is the image of $C$ under some continuous
mapping, in particular, there is a continuous surjection $h\colon C\to [0,1]^\infty$.
Then the induced mapping
$H\colon \mathcal{P}(C)\to \mathcal{P}([0,1]^\infty)$ defined by
$H(\mu)=\mu\circ h^{-1}$
is also a continuous surjection.
By the Milyutin theorem (see \cite[\S2.6]{B18} for details)
there is a continuous affine mapping
$G\colon \mathcal{P}([0,1]^\infty)\to \mathcal{P}(C)$ that is a right inverse for~$H$,
i.e., $H(G(\nu))=\nu$ for all $\nu\in \mathcal{P}([0,1]^\infty)$. Therefore,
using a jointly Borel measurable parametrization $\xi_\mu(t)$ for $\mathcal{P}(C)$ we obtain
a jointly Borel measurable parametrization $h(\xi_{G(\mu)}(t))$ for
$\mathcal{P}([0,1]^\infty)$. Hence the desired parametrization exists for every Borel
subspace in~$[0,1]^\infty$, but every Polish space is homeomorphic to a $G_\delta$-set
in~$[0,1]^\infty$, see \cite[Theorem 4.2.10, Theorem 4.3.24, Corollary 4.3.25]{Eng},
which completes the proof.
\end{proof}

\begin{remark}
\rm
A drawback of convergence almost everywhere is that there is no topology
in which convergent sequences are precisely the sequences converging almost everywhere.
For this reason it may be more convenient to consider on the space of Borel mappings from
$[0,1]$ to $X$ the semimetric of convergence in measure
defined by
$$
d_0(\xi,\eta)=\int_0^1 \min(d(\xi(t),\eta(t)),1)\, dt,
$$
where $d$ is a complete metric on~$X$. The corresponding quotient space is
also complete separable. It is clear that for the obtained parametrization
convergence of mappings in this semimetric is equivalent to weak convergence
of their laws.
Actually, this parametrization gives a homeomorphism
of the quotient space $L^0(\lambda,X)$ of $X$-valued mappings with convergence
in measure and the space~$\mathcal{P}(X)$.
\end{remark}

{\bf Acknowledgements.}
This work has been supported by the Russian Science Foundation
Grant 17-11-01058 at Lomonosov Moscow State University.
The results presented in Section~6 were obtained within the project of the second
author supported by the Foundation for the Advancement of Theoretical Physics and Mathematics ``BASIS''.
We are very grateful to Sergey Kuksin and Armen Shirikyan for inspiring
discussions and useful comments.


\begin{thebibliography}{10}

\bibitem{AY}
 G.A. Alekseev, E.V. Yurova,
On Gaussian conditional measures depending on a parameter, Theory Stoch. Processes 22 (2)
(2017), 1--7.

\bibitem{AG}
L. Ambrosio, N. Gigli,
A user's guide to optimal transport,
Lecture Notes in Math.  2062 (2013), 1--155.

\bibitem{AubinF}
J.-P. Aubin,  H. Frankowska,
Set-valued analysis.  Birkh\"auser Boston, Boston, 1990.

\bibitem{BBK01}
 T.O. Banakh, V.I. Bogachev, A.V. Kolesnikov,
Topological spaces with the strong Skorokhod property,
Georgian Math. J. 8 (2) (2001), 201--220.

\bibitem{BeigLS14}
 M. Beiglb\"ock,  C. Leonard, W. Schachermayer,
  On the duality theory for the Monge--Kantorovich transport problem,
   In:  Optimal transportation, pp.~216--265.
    London Math. Soc. Lecture Note Ser., V.~413.
     Cambridge Univ. Press, Cambridge, 2014.

\bibitem{BeigS}
 M. Beiglb\"ock, W. Schachermayer,
 Duality for Borel measurable cost functions,
 Trans. Amer. Math. Soc. 363 (8) (2011), 4203--4224.

 \bibitem{Black2}
 D. Blackwell, L.E. Dubins,
An extension of Skorohod's almost sure representation theorem,
Proc. Amer. Math. Soc. 89 (4) (1983), 691--692.

\bibitem{BR}
 D. Blackwell, C. Ryll-Nardzewski,
Non-existence of everywhere proper conditional distributions,
Ann. Math. Statist. 34 (1963), 223--225.

\bibitem{B07}
 V.I. Bogachev,
Measure Theory, vols.~1,~2, Springer, Berlin, 2007.

\bibitem{B10}
V.I. Bogachev,
Differentiable Measures and the Malliavin Calculus,
Amer. Math. Soc., Providence, Rhode Island, 2010.

\bibitem{B17}
 V.I. Bogachev, Surface measures in infinite-dimensional spaces, In:
 Measure theory in non-smooth spaces, pp.~52--97,
 Partial Differ. Equ. Meas. Theory, De Gruyter Open, Warsaw, 2017.

\bibitem{B18}
V.I. Bogachev, Weak Convergence of Measures, Amer. Math. Soc., Providence,
Rhode Island, 2018.

\bibitem{BoKol}
V.I. Bogachev, A.V. Kolesnikov,
Open mappings of probability measures and the Skorohod
representation theorem,
Teor. Veroyatn. Primen. 46 (1) (2001), 3--27 (in Russian);
English transl.:
Theory Probab. Appl. 46 (1) (2001), 20--38.

\bibitem{BK}
 V.I. Bogachev, A.V. Kolesnikov,
 The Monge--Kantorovich problem: achievements, connections, and prospects,
Uspekhi Matem. Nauk 67 (5) (2012), 3--110 (in Russian); English transl.:
Russian Math. Surveys 67 (5) (2012), 785--890.

\bibitem{BM}
V.I. Bogachev, I.I. Malofeev,  Surface measures generated by differentiable measures,
  Potential Anal. 44 (4) (2016), 767--792.

  \bibitem{CRFV}
 C. Castaing,  P. Raynaud de Fitte, M. Valadier,
Young Measures on Topological Spaces. With Applications in
Control Theory and Probability Theory, Kluwer, Dordrecht, 2004.

\bibitem{CV}
C. Castaing, M. Valadier,
Convex Analysis and Measurable Multifunctions, Lecture Notes in Math.
V.~580, Springer-Verlag, Berlin -- New York, 1977.

  \bibitem{DePR}
  J. Dedecker,  C. Prieur, P. Raynaud De Fitte,
 Parametrized Kantorovich--Rubin\v{s}tein theorem
 and application to the coupling of random variables,
 In: Dependence in probability and statistics,
 pp.~105--121, Lect. Notes Stat., V.~187, Springer, New York, 2006.

  \bibitem{Del}
C. Dellacherie,
 Un cours sur les ensembles analytiques,
In: Analytic sets, pp.~184--316. Academic Press, New York, 1980.

\bibitem{Doled}
A.N. Doledenok,  On a Kantorovich problem with a density constraint.
 Mat. Zametki 104 (1) (2018), 45--55 (in Russian);
English transl.: Math. Notes 104 (1) (2018), 39--47.

\bibitem{Eng}
P. Engelking, General Topology, Polish Sci. Publ., Warszawa, 1977.

\bibitem{Ev}
I.V. Evstigneev,
Regular conditional expectations of random variables depending on parameters,
Teor. Veroyatnost. i Primenen. 31 (3) (1986), 586--589 (in Russian);
 English transl.:  Theory Probab. Appl. 31 (3) (1987), 515--518.

\bibitem{Fernique88}
X. Fernique,
Un mod\`ele presque s\^ur pour la convergence en loi,
C.~R. Acad. Sci. Paris, S\'er.~1 306 (1988), 335--338.

\bibitem{GM}
W. Gangbo, R.J. McCann, The geometry of optimal transportation,
Acta Math. 177 (1996), 113--161.

\bibitem{HT}
P.-L. Hennequin, A. Tortrat,
Th\'eorie des Probabilit\'es et Quelques Applications, Masson et Gie, Paris, 1965.

\bibitem{HP}
 J. Hille,  D. Plachky, J. Roters,
Versions of conditional expectations depending continuously on parameters,
Math. Methods Statist.  8 (1) (1999), 99--108.

\bibitem{H}
J. Hoffmann-J{\o}rgensen,
Existence of conditional probabilities,
Math. Scand. 28 (2) (1971), 257--264.

\bibitem{H94}
J. Hoffmann-J{\o}rgensen, Probability with a View Toward Statistics, vols.~I,~II,
Chapman \& Hall, New York, 1994.

\bibitem{Kech}
 A.S. Kechris,
     Classical Descriptive Set Theory,
     Springer, Berlin -- New York, 1995.

     \bibitem{Kellerer84}
H.G. Kellerer,
Duality theorems for marginal problems, Z.~Wahrsch.
verw. Geb. 67 (4) (1984), 399--432.

 \bibitem{KM}
J. Korman, R.J. McCann,  Optimal transportation with capacity constraints,
Trans. Amer. Math. Soc. 367 (3) (2015), 1501--1521.

 \bibitem{KNS}
 S. Kuksin,   V. Nersesyan, A. Shirikyan,
Exponential mixing for a class of dissipative PDEs with bounded degenerate noise,
Arxiv 1802.03250v2.

\bibitem{Malofeev}
 I.I. Malofeev,  Measurable dependence of conditional measures on a parameter,
 Dokl. Akad. Nauk 470 (1) (2016), 13--17 (in Russian);
 English transl.: Dokl. Math. 94 (2) (2016), 493--497.

\bibitem{Pf}
J. Pfanzagl,
Parametric Statistical Theory,  Walter de Gruyter, Berlin, 1994.

 \bibitem{RR}
 S.T. Rachev, L. R{\"u}schendorf,
 Mass Transportation Problems, vols.~I,~II, Springer, New York, 1998.

\bibitem{Ram}
 D. Ramachandran,
A note on regular conditional probabilities in Doob's sense,
Annals Probab. 9 (5) (1981), 907--908.

\bibitem{Rao}
M.M. Rao,
Conditional Measures and Applications,
2nd ed. Chapman and Hall/CRC, Boca Raton, Florida, 2005.

\bibitem{Sk}
A.V. Skorohod,
Limit theorems for stochastic processes,
Teor. Veroyatn. Primen. 1 (1956), 261--290 (in Russian);
English transl.: Theory Probab. Appl. 1 (1956), 261--290.

\bibitem{Tjur}
T. Tjur,
Conditional Probability Distributions,
Lecture Notes, No.~2, Institute of Mathematical Statistics,
University of Copenhagen, Copenhagen, 1974.

\bibitem{Trumbo}
B.E. Trumbo,
Sufficient conditions for the weak
 convergence of conditional probability distributions in a metric space,
Thesis (Ph.D.) The University of Chicago, 1965.

\bibitem{V03}
 C. Villani,   Topics in Optimal Transportation,
Amer. Math. Soc., Providence, Rhode Island, 2003.

\bibitem{V}
C. Villani,  Optimal Transport, Old and New, Springer, New York, 2009.

\bibitem{Zhang}
X. Zhang,  Stochastic Monge--Kantorovich problem and its duality,
           Stochastics 85 (1) (2013), 71--84.

\end{thebibliography}
\end{document}